\documentclass[11pt]{amsart}

\usepackage{amsmath, amssymb, amscd, cancel, graphicx, soul, stmaryrd}

\usepackage{mathdots}

\usepackage[all]{xy}
\usepackage{multirow}
\usepackage{longtable}
\usepackage{array}

\newtheorem{thm}{Theorem}[section]
\newtheorem{conj}[thm]{Conjecture}
\newtheorem{cor}[thm]{Corollary}
\newtheorem{lem}[thm]{Lemma}
\newtheorem{prop}[thm]{Proposition}

\newtheorem{defin}[thm]{Definition}

\def\cC{{\mathcal C}}
\def\cF{{\mathcal F}}

\def\s{{\mathfrak s}}

\def\bF{{\mathbb F}}

\def\O{{\mathcal O}}

\def\Q{{\mathbb Q}}
\def\bR{{\mathbb R}}

\def\Z{{\mathbb Z}}

\def\L{{\langle}}
\def\R{{\rangle}}

\def\Char{{\mathrm{Char}}}

\def\spc{{\mathrm{spin^c}}}
\def\Spc{{\mathrm{Spin^c}}}
\def\Short{{\mathrm{Short}}}

\def\cut{{\textup{cut}}}
\def\cyc{{\textup{cyc}}}
\def\del{{\partial}}
\def\disc{{\textup{disc}}}
\def\im{{\text{im}}}
\def\mod{{\textup{mod} \;}}
\def\rk{{\mathrm{rk}}}
\def\span{{\textup{span}}}
\def\supp{{\textup{supp}}}

\newcommand{\into}{\hookrightarrow}

\begin{document}

\title[Lattices, graphs, and Conway mutation]%
{Lattices, graphs, and Conway mutation}

\author[Joshua Evan Greene]{Joshua Evan Greene}

\address{Department of Mathematics, Columbia University\\ New York, NY 10027}

\thanks{Partially supported by an NSF Postdoctoral Fellowship.}

\email{josh@math.columbia.edu}

\maketitle

\begin{abstract}

The $d$-invariant of an integral, positive definite lattice $\Lambda$ records the minimal norm of a characteristic covector in each equivalence class $(\mod 2\Lambda)$.   We prove that the 2-isomorphism type of a connected graph is determined by the $d$-invariant of its lattice of integral cuts (or flows).  As an application, we prove that a reduced, alternating link diagram is determined up to mutation by the Heegaard Floer homology of the link's branched double-cover.  Thus, alternating links with homeomorphic branched double-covers are mutants.

\end{abstract}





\section{Introduction.}

Conway mutation has been in the news a lot lately.  Given a sphere $S^2$ that meets a link $L \subset S^3$ transversely in four points, cut along it and reglue by an involution that fixes a pair of points disjoint from $L$ and permutes $S^2 \cap L$.  This process results in a new link $L' \subset S^3$, and a pair of links are called {\em mutants} if they are related by a sequence of such transformations.  An analogous definition of Conway mutation applies to link diagrams.


A fundamental question about any link invariant is whether it can distinguish mutants.  One such invariant is the homeomorphism type of the space $\Sigma(L)$, the double-cover of $S^3$ branched along $L$.  As first noted by Viro, mutant links possess homeomorphic branched double-covers \cite[Thm.1]{viro:mutation}, \cite[Prop.3.8.2]{kawauchi:book}.  It follows that any invariant of branched double-covers will not distinguish mutants either.  However, non-mutant links can possess homeomorphic branched double-covers, such as the pretzel knot $P(-2,3,7)$ and the torus knot $T(3,7)$.  It remains an intriguing open problem to classify distinct links with homeomorphic branched double-covers \cite[Probs.1.22\&3.25]{kirby:problems}.

Our purpose here is to show that within the class of {\em alternating} links, the Heegaard Floer homology of the branched double-cover provides a {\em complete} invariant for the mutation type.

\begin{thm}\label{t: main}

Given a pair of connected, reduced alternating diagrams $D, D'$ for a pair of links $L,L'$, the following assertions are equivalent:

\begin{enumerate}

\item $D$ and $D'$ are mutants;

\item $L$ and $L'$ are mutants;

\item $\Sigma(L) \cong \Sigma(L')$; and

\item $\widehat{HF}(\Sigma(L)) \cong \widehat{HF}(\Sigma(L'))$, as absolutely graded, relatively $\spc$-graded groups.

\end{enumerate}

\end{thm}

\noindent 

We derive Theorem \ref{t: main} from Theorem \ref{t: main graph preview} below, a combinatorial result.  We proceed to sketch the main line of argument and then discuss some repercussions of Theorem \ref{t: main}.


\subsection{From topology to combinatorics.}

As indicated in the abstract, the focus of the paper is primarily combinatorial in nature.  This is thanks to a description of the invariant $\widehat{HF}(\Sigma(L))$ for an alternating link $L$ due to Ozsv\'ath-Szab\'o, which we quickly review.

First, the manifold $\Sigma(L)$ is an L-space.  This means that the invariant $\widehat{HF}$ has rank one in each $\spc$ structure on $\Sigma(L)$, the set of which forms a torsor over $H^2(\Sigma(L);\Z)$.  Thus, the invariant is completely captured by its Heegaard Floer $d$-invariant, which for the case at hand is the mapping $d: \Spc(\Sigma(L)) \to \Q$ that records the absolute grading in which each group is supported.  

To express the $d$-invariant, choose a reduced, alternating diagram $D$ for $L$, and let $G$ denote its Tait graph.  Associated to $G$ is its lattice of integral flows $\cF(G)$; this lattice is presented by the Goeritz matrix for $D$.  For an integral lattice $\Lambda$, define the characteristic coset
\[ \Char(\Lambda) = \{ \chi \in \Lambda^* \; | \; \L \chi, y \R \equiv |y| \, (\mod 2), \forall y \in \Lambda \}, \]
where $\L \, , \, \R$ denotes the inner product and $|\cdot|$ the norm (self-pairing) of an element.  The set $C(\Lambda) = \Char(\Lambda) \, (\mod 2 \Lambda)$ forms a torsor over the discriminant group $\Lambda^* / \Lambda$.  Given $[\chi] \in C(\Lambda)$, define
\[ d_\Lambda([\chi]) = \min \left\{ \frac{|\chi'|- \rk(\Lambda)}{4} \; | \; \chi' \in [\chi] \right\},\] 
and call an element $\chi \in \Char(\Lambda)$ {\em short} if its norm is minimal in $[\chi]$.  We call the pair $(C(\Lambda),d_\Lambda)$ the {\em $d$-invariant} of the lattice $\Lambda$.  

There exists a natural identification between the torsors $\Spc(\Sigma(L))$ and $C(\cF(G))$, and Ozsv\'ath-Szab\'o showed that this identification extends to an isomorphism between the pairs $(\Spc(\Sigma(L)),d)$ and $(C(\cF(G)),-d_\cF)$ (Theorem \ref{t: osz}).  In summary, the isomorphism type of $\widehat{HF}(\Sigma(L))$ is determined by the (lattice theoretic) $d$-invariant of the lattice of integral flows on the Tait graph.

The foregoing description of $\widehat{HF}(\Sigma(L))$ begs the question: when do the flow lattices attached to a pair of graphs have isomorphic $d$-invariants?  Our main combinatorial result answers this question.

\begin{thm}\label{t: main graph preview}

The following are equivalent for a pair of 2-edge-connected graphs $G, G'$:

\begin{enumerate}

\item $\cF(G)$ and $\cF(G')$ have isomorphic $d$-invariants;

\item $\cF(G) \cong \cF(G')$; and

\item $G$ and $G'$ are 2-isomorphic.

\end{enumerate}

\end{thm}

\noindent A {\em 2-isomorphism} between a pair of 2-edge-connected graphs is a cycle-preserving bijection between their edge sets.  Note that Theorem \ref{t: main graph preview} applies to arbitrary 2-edge-connected graphs, not just planar ones. Also, the implication (3)$\implies$(2) appears in \cite[Prop.5]{bacheretal:lattices}, and (2)$\implies$(3) resolves the question implicit at the end of that paper.  An extended version of Theorem \ref{t: main graph preview} appears as Theorem \ref{t: main graph} below. 


\subsection{Prospectus on Theorem \ref{t: main}.}

We sketch the proof of Theorem \ref{t: main}, using Theorem \ref{t: main graph preview}.  The forward implications in Theorem \ref{t: main} are immediate, so it stands to establish (4)$\implies$(1).  Thus, choose a pair of reduced, alternating diagrams $D,D'$ for a pair of links $L,L'$ for which $\widehat{HF}(\Sigma(L)) \cong \widehat{HF}(\Sigma(L'))$.  It follows that $\cF(G)$ and $\cF(G')$ have isomorphic $d$-invariants, where $G, G'$ denote the Tait graphs.  By Theorem \ref{t: main graph preview}, it follows that $G$ and $G'$ are 2-isomorphic.  Now we invoke two graph theoretic results. First, a theorem of Whitney asserts that a pair of 2-isomorphic graphs are related by a sequence of switches. Second, using Whitney's result, a theorem of Mohar-Thomassen about planar graphs asserts that any two planar drawings of a pair of 2-isomorphic planar graphs are related by a sequence of flips, planar switches, and swaps.  Each of these transformations of planar graphs corresponds to a Conway mutation of link diagrams, so it follows that $D, D'$ are mutants.


\subsection{Prospectus on Theorem \ref{t: main graph preview}.}

Now we sketch the proof of Theorem \ref{t: main graph preview}.  To a graph $G$ we associate the chain group $C_1(G;\Z)$.  This group naturally inherits the structure of a lattice by taking the edge set of $G$ as an orthonormal basis.  Within $C_1(G;\Z)$ sits a pair of distinguished sublattices, the lattice of integral cuts $\cC(G)$, and the aforementioned lattice of integral flows $\cF(G)$.  (For the case of a planar graph $G$ with planar dual $G^*$, we have $\cC(G) \cong \cF(G^*)$.)  In general, $\cC(G)$ and $\cF(G)$ are complementary, primitive sublattices of $C_1(G;\Z)$.  Furthermore, every short characteristic covector for $\cC(G)$ and $\cF(G)$ is the restriction of one for $C_1(G;\Z)$.
It follows that the $d$-invariants of these sublattices are opposite one another: that is, there exists a natural isomorphism $(C(\cF(G)),d_\cF) \overset{\sim}{\to} (C(\cC(G)), -d_\cC)$. 

Now suppose that the flow lattices of $G$ and $G'$ have isomorphic $d$-invariants.  Since the discriminant groups are isomorphic, we can glue $\cF(G)$ and $\cC(G')$ to produce an integral, positive definite, unimodular lattice $\Lambda$.  Furthermore, since they have opposite $d$-invariants, $\Lambda$ has vanishing $d$-invariant.  By a theorem of Elkies, it follows that $\Lambda$ admits an orthonormal basis.  Using the fact that every short characteristic covector for $\cF(G)$ and $\cC(G')$ is the restriction of one for $\Lambda$, we can set the orthonormal basis for $\Lambda$ in one-to-one correspondence with the edge sets of $G$ and $G'$.  It easily follows that the resulting bijection between the edge sets of $G$ and $G'$ is a 2-isomorphism.


\subsection{Repercussions of Theorem \ref{t: main}.}\label{ss: repercussions}

Since diagrammatic mutations clearly preserve the number of crossings, Theorem \ref{t: main} implies that two reduced, alternating diagrams for the {\em same} link have the same number of crossings.  Furthermore, if a reduced, alternating diagram admits no non-trivial mutation, then it is the unique reduced, alternating diagram representing its link type (cf. \cite{schrijver:tait}).  Of course, much more is known now: any minimal crossing diagram of an alternating link is alternating \cite{kauff:altknots, murasugi:altknots, thistle:spanning}, and any two such diagrams are related by a sequence of {\em flypes} \cite{mt:tait}.  We simply point out that we obtain these corollaries in a rather different manner from how they were originally deduced, using graphs, lattices, and Floer homology in place of the Jones polynomial and explicit geometric arguments.

Second, Theorem \ref{t: main} generalizes the homeomorphism classification of the three-dimensional lens spaces and the isotopy classification of two-bridge links.  The lens spaces arise as the branched double-covers of the two-bridge links, and we argue directly in Proposition \ref{p: 2bridge} that a pair of two-bridge diagrams in standard position are mutants iff they coincide up to isotopy and reversal.  Using this fact, Theorem \ref{t: main} implies a one-to-one correspondence between such diagrams and lens spaces, yielding at once the classification of both.  We point out that $\widehat{HF}$ recovers the Reidemeister torsion by a result of Rustamov \cite[Thm.3.4]{rustamov:turaev}, which is well-known to completely distinguish the homeomorphism types of lens spaces \cite{brody:torsion, reidemeister:torsion}, and which leads to the classification of 2-bridge links \cite{schubert:2bridge}.  Of course, the homeomorphism classification follows from the stronger result that every lens space possesses a unique Heegaard torus up to isotopy \cite[Thm.1]{bonahon:lens}, \cite[Thm.5.1]{hr:2bridge}.  We simply point out, once again, the different manner of our argument.

Third, note that Theorem \ref{t: main} cannot extend too far beyond the domain of alternating links, due to the existence of non-mutant links with homeomorphic branched double-covers.  It would be interesting to know whether Theorem \ref{t: main} generalizes to quasi-alternating links.  Note that the invariant $\widehat{HF}$ does not distinguish the branched double-covers of alternating and non-alternating knots in general, such as the unknot and $T(3,5) \# \overline{T(3,5)}$.  However, we propose the following conjecture.

\begin{conj}\label{conj: funny pair}

There does not exist a pair of links, one alternating, the other non-alternating, with homeomorphic branched double-covers.

\end{conj}

We provide some limited evidence in support of Conjecture \ref{conj: funny pair}.  First, no mutant pair violates Conjecture \ref{conj: funny pair}, since a result of Menasco implies that mutation preserves alternatingness \cite[Proof of Thm.3(b)]{menasco:alternating}.  Second, Hodgson-Rubinstein showed that a two-bridge link is uniquely determined by its branched double-cover \cite[Cor.4.12]{hr:2bridge}.  Third, Dunfield investigated knots with at most 16 crossings whose branched double-covers are hyperbolic and of small enough volume to appear in the Hodgson-Weeks census of closed hyperbolic 3-manifolds.  He reports 3765 non-alternating knots with such branched covers but only 178 alternating knots, and no manifold appears as the branched double-cover of both kinds of knots \cite{dunfield:comm}.  Finally, and most persuasively to this author, is the lack of any counterexample known to the (non-exhaustive) list of experts we consulted!


\subsection{Organization.}

The main body is organized into three sections: {\em Lattices}, {\em Graphs}, and {\em Conway Mutation}.  Each section begins at a basic level and invokes a key auxiliary result: Elkies's theorem on unimodular lattices; Whitney's theorem on 2-isomorphism of graphs; and Ozsv\'ath-Szab\'o's theorem on $\widehat{HF}(\Sigma(L))$, respectively.  It remains a curious fact that although Elkies's theorem asserts a purely algebraic fact and we it use towards a combinatorial end, its only known proof relies on analytical methods (modular forms).  Lastly, it is perhaps fitting that the chief insight involved the use of lattice gluing, a technique we learned from Conway, to establish the desired result about Conway mutation of alternating links.


\section*{Acknowledgments.}

Thanks to Nathan Dunfield and Oleg Viro for helpful correspondence, to Liam Watson for an inspiring conversation, and to the Venice Beach Institute for its hospitality. 


\section{Lattices.}


\subsection{Preparation.}

A {\em lattice} consists of a finitely-generated, free abelian group $\Lambda$ equipped with a non-degenerate, symmetric bilinear form \[ \L \, , \, \R : \Lambda \times \Lambda \to \Q.\]  It is {\em integral} if the image of its pairing lies in $\Z$, and the symbol $\Lambda$ will always denote an integral lattice in what follows.  The form extends to a $\Q$-valued pairing on $\Lambda \otimes \Q$, which allows us to define the {\em dual lattice}
\[ \Lambda^* := \{ x \in \Lambda \otimes \Q \; | \; \L x,y \R \in \Z, \; \forall \, y \in \Lambda \}.\] 
Given $x \in \Lambda^*$, denote by $\overline{x}$ its image in the {\em discriminant group} $\overline{\Lambda} := \Lambda^*/\Lambda$.  The  {\em discriminant} $\disc(\Lambda)$ is the order of this finite group, and $\Lambda$ is {\em unimodular} if $\disc(\Lambda)=1$.  For example, the lattice generated by $n$ orthonormal elements is the integral, unimodular lattice denoted $\Z^n$.  The pairing on $\Lambda$ descends to a non-degenerate, symmetric bilinear form
\[ b : \overline{\Lambda} \times \overline{\Lambda} \to \Q / \Z,\]
\[ b(\overline{x},\overline{y}) \equiv \L x,y \R \; (\mod 1), \]
the {\em discriminant form} (or, for topologists, the {\em linking form}).  The {\em norm} of $x \in \Lambda^*$ is the self-pairing $|x| := \L x,x \R$, and we set $q(\overline{x}) := b(\overline{x},\overline{x})$.


\subsection{The characteristic coset.}

Let
\[ \Char(\Lambda) = \{ \chi \in \Lambda^* \; | \; \L \chi,y \R \equiv |y| \, (\mod 2), \, \forall \, y \in \Lambda \} \]
denote the set of {\em characteristic covectors} for $\Lambda$.  This set constitutes a distinguished coset in $\Lambda^* / 2\Lambda^*$.  Correspondingly, the set \[ C(\Lambda) := \Char(\Lambda) \; (\mod 2\Lambda) \] forms a torsor over the group $2\Lambda^* / 2\Lambda \cong \overline{\Lambda}$.  Thus, for a unimodular lattice, such as $\Z^n$, this torsor has one element.  Given $\chi \in \Char(\Lambda)$, let $[\chi]$ denote its image in $C(\Lambda)$.  We never refer to $\overline{\chi} \in \overline{\Lambda}$, so no confusion should result.  We obtain a map
\[ \rho: C(\Lambda) \to \Q / 2 \Z, \]
\[ \rho([\chi]) \equiv \frac{|\chi| - \sigma(\Lambda)}{4} \; (\mod 2),\]
where $\sigma(\Lambda)$ denotes the {\em signature} of the pairing on $\Lambda$ (cf. \cite{os:definite}, \cite[\S 1.1]{os:absgr}).  The map $\rho$ is well-defined since
\[ \frac{1}{4}(|\chi + 2y| - |\chi|) = \L \chi, y \R + |y| \equiv 0 \, (\mod 2), \quad \forall y \in \Lambda.\]
By definition, the pair $(C(\Lambda),\rho)$ is the {\em $\rho$-invariant} of $\Lambda$.  In this terminology, we have the following classical result.

\begin{thm}[van der Blij \cite{vanderblij:characteristic}]\label{t: van der blij}

The mapping $\rho: C(\Lambda) \to \Q / 2\Z$ vanishes for a unimodular lattice $\Lambda$. \qed

\end{thm}

\begin{defin}\label{d: torsor isomorphism}

Let $f_i : C_i \to S$ be a map from a $G_i$-torsor $C_i$ to a set $S$, for $i=1,2$.  An {\em isomorphism} \[ \varphi: (C_1,f_1)  \overset{\sim}{\to} (C_2,f_2)\] consists of a bijection $\varphi : C_1 \to C_2$ and a group isomorphism $ \varphi : G_1 \to G_2$ such that $ f_2 = \varphi \circ f_1 $ and $\varphi(c)-\varphi(c') = \varphi(c-c')$ for all $c,c' \in C_1$.

\end{defin}

\begin{lem}[cf. \cite{os:definite}, Prop.6]\label{l: rho => disc}

An isomorphism of $\rho$-invariants $\varphi: (C(\Lambda_1),\rho_1) \to (C(\Lambda_2),\rho_2)$ induces an isomorphism of discriminant forms $\varphi: (\overline{\Lambda}_1,b_1) \to (\overline{\Lambda}_2,b_2)$.

\end{lem}

\begin{proof}

Given a lattice $\Lambda$, fix $\chi \in \Char(\Lambda)$ and select $x,y \in \Lambda^*$.  From the identity
\[ \L x,y \R = \frac{1}{8} ( |\chi+2x+2y| + |\chi| - |\chi+2x| + |\chi+2y| ) \]
we obtain
\[ b(\overline{x},\overline{y}) \equiv \frac{1}{2} ( \rho([\chi]+2\overline{x}+2\overline{y}) + \rho([\chi]) - \rho([\chi]+2\overline{x}) - \rho([\chi]+2\overline{y})) \, (\mod 1). \]
The statement of the Lemma now follows directly.

\end{proof}


\subsection{Gluing.}\label{ss: gluing}

Given a sublattice $\Lambda_1$ of a lattice $\Lambda$, we obtain a natural restriction map
\[ r_1: \Lambda^* \to \Lambda_1^*.\]
The sublattice $\Lambda_1 \subset \Lambda$ is {\em primitive} if the mapping $r_1$ surjects.  A pair of sublattices $\Lambda_1, \Lambda_2 \subset \Lambda$ are {\em complementary} if they are orthogonal and their ranks sum to that of $\Lambda$.  One may check that $\Lambda_1, \Lambda_2 \subset \Lambda$ are primitive and complementary iff $\Lambda_1^\perp = \Lambda_2$ and $\Lambda_2^\perp = \Lambda_1$, although we shall not need this fact.  Versions of the following gluing Lemma have been observed by a number of researchers, e.g. \cite[Ch.4, Thm.1]{cs:lattices}, \cite[Prop.1]{os:definite}.

\begin{lem}\label{l: pre-gluing 1}

Suppose that $\Lambda_1, \Lambda_2$ are a pair of complementary, primitive sublattices of a unimodular lattice $\Lambda$.  Then there exists a natural isomorphism
\begin{equation}\label{e: phi b}
\varphi: (\overline{\Lambda}_1,b_1) \overset{\sim}{\to} (\overline{\Lambda}_2,-b_2).
\end{equation}
Conversely, suppose that $\Lambda_1$ and $\Lambda_2$ are a pair of integral lattices and there exists an isomorphism $\varphi$ as in \eqref{e: phi b}. Then the glue lattice
\[ \Lambda_1 \oplus_\varphi \Lambda_2 := \{ x + y \in \Lambda_1^* \oplus \Lambda_2^* \; | \; \varphi(\overline{x}) = \overline{y} \} \]
is an integral, unimodular lattice that contains $\Lambda_1$ and $\Lambda_2$ as complementary, primitive sublattices.

\end{lem}

\begin{proof}

($\implies$)To quote \cite{cs:lattices}, the isomorphism is given by $\varphi(\overline{x}) =\overline{y}$ whenever $x+y \in \Lambda$.  Verification of the stated properties is straightforward.

\noindent ($\impliedby$) By construction, $\Lambda := \Lambda_1 \oplus_\varphi \Lambda_2$ is a lattice that contains $\Lambda_1$ and $\Lambda_2$ as complementary sublattices.  It is integral since, given $x+y \in \Lambda$, we have
\[ |x+y| = |x| + |y| \equiv q_1(\overline{x}) + q_2(\overline{y}) \equiv q_1(\overline{x})+q_2(\varphi(\overline{x})) \equiv q_1(\overline{x}) - q_1(\overline{x}) \equiv 0 \, (\mod 1).\]
It is unimodular since $\Lambda / (\Lambda_1 \oplus \Lambda_2)  \cong \{ (\overline{x},\varphi(\overline{x})) \in \overline{\Lambda}_1 \oplus \overline{\Lambda}_2 \}$ is a square-root order subgroup of $\overline{\Lambda}_1 \oplus \overline{\Lambda}_2 \cong (\Lambda_1 \oplus \Lambda_2)^*/ (\Lambda_1 \oplus \Lambda_2)$.  Since $\Lambda$ is integral and unimodular, the restriction maps $r_1, r_2$ are simply the projections $\Lambda \to \Lambda_1^*$, $\Lambda \to \Lambda_2^*$.  These maps surject by construction, so $\Lambda_1$ and $\Lambda_2$ are primitive sublattices of $\Lambda$.

\end{proof}

The construction of Lemma \ref{l: pre-gluing 1} behaves well with respect to characteristic cosets.

\begin{lem}\label{l: pre-gluing 2}

Suppose that $\Lambda_1, \Lambda_2$ are a pair of complementary, primitive sublattices of a unimodular lattice $\Lambda$.  Then there exists a natural isomorphism
\begin{equation}\label{e: phi rho}
\varphi: (C(\Lambda_1),\rho_1) \overset{\sim}{\to} (C(\Lambda_2),-\rho_2).
\end{equation}
Conversely, suppose that $\Lambda_1$ and $\Lambda_2$ are a pair of integral lattices and there exists an isomorphism $\varphi$ as in \eqref{e: phi rho}.  Then $\Lambda_1 \oplus_\varphi \Lambda_2$ is an integral, unimodular lattice that contains $\Lambda_1$ and $\Lambda_2$ as complementary, primitive sublattices.

\end{lem}

\noindent The mapping of discriminant groups used in the gluing $\Lambda_1 \oplus_\varphi \Lambda_2$ is the one implicit in the torsor map $\varphi$, as in Definition \ref{d: torsor isomorphism}.

\begin{proof}

($\implies$) To begin with, observe that
\begin{equation}\label{e: intersection}
2\Lambda \cap (2\Lambda_1 + \Lambda^*_2) = 2\Lambda_1 + 2\Lambda_2
\end{equation}
(we identify $\Lambda_1^\perp \subset \Lambda$ with $\Lambda_2^*$ using Lemma \ref{l: pre-gluing 1}).  Indeed, if $2(x+y) \in 2L$ with $2x \in 2\Lambda_1$, $2y \in \Lambda^*_2$, then $0 = \overline{x} \in \overline{\Lambda}_1$, so $0 = \varphi(\overline{x}) = \overline{y} \in \overline{\Lambda}_2$ using the isomorphism of Lemma \ref{l: pre-gluing 1}.  Thus, $y \in \Lambda_2$ and $2y \in 2\Lambda_2$, as desired.

Next, each restriction map $r_i$ clearly carries $\Char(\Lambda)$ onto a subset of $\Char(\Lambda_i)$.  Furthermore, since $r_i$ maps $\Lambda$ onto $\Lambda_i^*$, it carries $ 2\Lambda$ onto $2\Lambda_i^*$, and hence the coset $\Char(\Lambda)$  {\em onto} $\Char(\Lambda_i)$.  Thus, given a pair of elements $\chi_1,\chi'_1 \in \Char(\Lambda_1)$ with $[\chi_1] = [\chi'_1] \in C(\Lambda_1)$, there exists a pair of elements $\chi_2, \chi'_2 \in \Char(\Lambda_2)$ such that $\chi = \chi_1 + \chi_2$ and $\chi' = \chi'_1 + \chi'_2$ belong to $\Char(\Lambda)$.  Their difference $\chi-\chi'$ belongs to $2\Lambda$ since $\Lambda$ is unimodular, and $\chi_1 - \chi'_1 \in 2\Lambda_1$ by assumption.  It follows from \eqref{e: intersection} that $[\chi_2] = [\chi'_2] \in C(\Lambda_2)$.  Thus, we obtain a well-defined mapping
\[ \varphi: C(\Lambda_1) \overset{\sim}{\to} C(\Lambda_2),\]
\[\varphi([\chi_1]) = [\chi_2] \iff \chi_1 + \chi_2 \in \Char(\Lambda), \]
whence
\[\Char(\Lambda) = \{ \chi_1 + \chi_2 \in \Char(\Lambda_1) \oplus \Char(\Lambda_2) \; | \; \varphi([\chi_1]) = [\chi_2] \}.\]
Furthermore, if $\varphi([\chi_1]) = [\chi_2]$, then
\[ \rho([\chi_1]) + \rho([\chi_2]) \equiv \frac{|\chi_1| -\sigma(\Lambda_1)}{4} + \frac{|\chi_2| -\sigma(\Lambda_2)}{4} = \frac{|\chi_1 + \chi_2| - \sigma(\Lambda)}{4} \equiv \rho(\Lambda) \equiv 0 \, (\mod 2), \]
applying Theorem \ref{t: van der blij} at the last step.  Finally, the mapping $\varphi$ covers the isomorphism of discriminant forms from Lemma \ref{l: pre-gluing 1}, so it preserves the torsor structure.  This establishes the first part of the Lemma.

\noindent ($\impliedby$) This follows directly on combination of Lemma \ref{l: rho => disc} and the second part of Lemma \ref{l: pre-gluing 1}.

\end{proof}


\subsection{The positive definite case.}

When the form $\L \, , \, \R$ on $\Lambda$ is {\em positive definite}, its rank $n$ equals its signature $\sigma(\Lambda)$, and we obtain a $\Q$-valued lift of the $\rho$-invariant by defining
\[ d: C(\Lambda) \to \Q, \]
\[ d([\chi]) = \min \left\{ \frac{ |\chi'| - n}{4} \; | \; \chi' \in [\chi] \right\}.\]
By definition, the pair $(C(\Lambda),d)$ is the {\em $d$-invariant} $d(\Lambda)$ of the positive definite lattice $\Lambda$.  It is clearly additive, in the sense that there exists a natural isomorphism
\[ (C(\Lambda), d) \overset{\sim}{\to} (C(\Lambda_1) \oplus C(\Lambda_2), d_1 + d_2) \]
whenever $\Lambda \cong \Lambda_1 \oplus \Lambda_2$.  We further define
\[ \Short(\Lambda) = \{ \chi \in \Char(\Lambda) \: | \: |\chi| \leq |\chi'|, \, \forall \chi' \in [\chi] \}, \]
and refer to elements of $\Short(\Lambda)$ as {\em short} characteristic covectors.  For example, 
\[ \Short(\Z^n) = \{ \chi \; | \; \L \chi, e_i \R = \pm 1 \; \forall i \},\]
where $\{e_1,\dots,e_n\}$ denotes an orthonormal basis for $\Z^n$.  Thus, $|\chi| = n$ for all $\chi \in \Short(\Z^n)$, so the mapping $d: C(\Z^n) \to \Q$ is zero.  Conversely, we have the following fundamental result.

\begin{thm}[Elkies \cite{elkies}]\label{t: elkies}

If $\Lambda$ is a rank $n$, unimodular, integral, positive definite lattice and $|\chi| \geq n$ for all $\chi \in \Char(\Lambda)$, then $\Lambda \cong \Z^n$, i.e. $\Lambda$ admits an orthonormal basis. \qed

\end{thm}

Observe that in the construction of $\S$\ref{ss: gluing}, each map $r_i$ restricts to a map
\begin{equation}\label{e: short map}
\Short(\Lambda) \to \Short(\Lambda_i),
\end{equation}
where $\Lambda = \Lambda_1 \oplus_\varphi \Lambda_2$.  Indeed, given $\chi_1 + \chi_2 \in \Short(\Lambda)$ and any $\chi'_i \in \Char(\Lambda_i)$ with $\chi'_i \in [\chi_i]$, $i=1,2$, we have $\chi'_1 + \chi'_2 \in \Char(\Lambda)$, and
\[ |\chi'_1| + |\chi'_2| = |\chi'_1+\chi'_2| \geq |\chi_1 + \chi_2| = |\chi_1| + |\chi_2|, \]
which shows that $|\chi'_i| \geq |\chi_i|$, $i=1,2$.   In general, the restiction \eqref{e: short map} need not surject.  For example, if
\[ \Lambda_1 = \span(e_1 + 2e_2), \, \Lambda_2 = \span(-2e_1+e_2) \subset \Z^2,\]
then $|\Short(\Z^2)| =4$, whereas $|\Short(\Lambda_1)| \geq |C(\Lambda_1)| = \disc(\Lambda_1) = 5$.  However, it is clear that $\Short(\Lambda) \to \Short(\Lambda_1)$ surjects iff $\Short(\Lambda) \to \Short(\Lambda_2)$ does.

Using Elkies's Theorem, we obtain the following refinement of Lemma \ref{l: pre-gluing 2}.

\begin{prop}\label{p: gluing}

Suppose that $\Lambda_1, \Lambda_2$ are a pair of complementary, primitive sublattices of $\Z^n$, and the restriction $\Short(\Z^n) \to \Short(\Lambda_1)$ surjects.
Then there exists a natural isomorphism
\begin{equation}\label{e: phi d}
\varphi: (C(\Lambda_1),d_1) \overset{\sim}{\to} (C(\Lambda_2),-d_2).
\end{equation}
Conversely, suppose that $\Lambda_1$ and $\Lambda_2$ are a pair of integral lattices and there exists an isomorphism as in \eqref{e: phi d}.  Then $\Lambda_1 \oplus_\varphi \Lambda_2 \cong \Z^n$, and the restrictions $\Short(\Z^n) \to \Short(\Lambda_i)$ surject.

\end{prop}

\begin{proof} Set $n_i = \rk(\Lambda_i)$, $i=1,2$.

\noindent ($\implies$) We use the isomorphism $\varphi$ of Lemma \ref{l: pre-gluing 2}.  Suppose that $\chi_1$ is a short representative for its class in $C(\Lambda_1)$.  Since $\Short(\Z^n) \to \Short(\Lambda_1)$ surjects, we have $\chi_1 + \chi_2 \in \Short(\Z^n)$ for some $\chi_2 \in \Short(\Lambda_2)$.  Since
\begin{equation}\label{e: short}
\frac{|\chi_1+\chi_2| -n}{4} = \frac{|\chi_1|-n_1}{4} + \frac{|\chi_2|-n_2}{4} = d_1([\chi_1])+d_2([\chi_2]) = d_1([\chi_1]) + d_2(\varphi([\chi_2]))
\end{equation}
and the left-most term is zero, it follows that $\varphi$ yields the desired isomorphism.

\noindent ($\impliedby$) Select any $\chi \in \Short(\Lambda_1 \oplus_\varphi \Lambda_2)$.  Then $\chi = \chi_1 + \chi_2$ with $\varphi([\chi_1]) = [\chi_2]$ and $\chi_i \in \Short(\Lambda_i)$, $i=1,2$.  Again, \eqref{e: short} holds, and now the right-most term is zero, so $|\chi| = n$ and $\Lambda_1 \oplus_\varphi \Lambda_2 \cong \Z^n$ by Theorem \ref{t: elkies}.  Furthermore, if $\chi_1 \in \Short(\Lambda_1)$ is given and $\chi_2 \in \Short(\Lambda_2)$ is chosen so that $\varphi([\chi_1]) = [\chi_2]$, then \eqref{e: short} applies again to show that $\chi_1 + \chi_2 \in \Short(\Z^n)$.  Hence the restriction $\Short(\Z^n) \to \Short(\Lambda_1)$ surjects.

\end{proof}


\subsection{Rigid embeddings.}

The following Proposition establishes a condition under which a lattice admits an essentially unique embedding into $\Z^n$.  It plays a key role in the proof of Theorem \ref{t: main graph}.
\begin{prop}\label{p: rigid}

Let $\Lambda$ denote a lattice, $B_\Lambda$ a basis for $\Lambda$, $Z_i$ a lattice with orthonormal basis $B_i$, and  $\iota_i : \Lambda \into Z_i$ an embedding, $i=1,2$.  Suppose that $\iota_1$ has the property that
\begin{equation}\label{e: 0/1}
\L \iota_1(x), e \R \in \{0,1\}, \quad \forall x \in B_\Lambda, e \in B_1,
\end{equation}
and
\begin{equation}\label{e: full support}
\forall \, e \in B_1, \exists \, x \in B_\Lambda \quad s.t. \quad \L\iota_i(x), e \R \ne 0;
\end{equation}
and suppose that both restriction maps $Z_i^* \to \Lambda^*$ induce surjections
\[ r_i : \Short(Z_i) \to \Short(\Lambda). \]
Then there exists an embedding $\iota: Z_1 \into Z_2$ such that $\iota_2 = \iota \circ \iota_1$.

\end{prop}

We work towards the proof of Proposition \ref{p: rigid} through a sequence of Lemmas.
Define
\[ \supp^\pm(x) = \{e \in B_i\; | \; \pm \L x, e \R > 0 \}, \quad \supp(x) = \supp^+(x) \cup \supp^-(x), \quad \forall x \in Z_i, i=1,2.\]
Hence  $\supp^-(\iota_1(x)) = \varnothing$ and $|x| = |\supp(\iota_1(x))|$, $\forall x \in B_\Lambda$.  Note as well that
\[ \L \chi, \iota_i(x) \R = \L r_i(\chi), x \R, \quad \forall x \in B_\Lambda, \chi \in \Short(Z_i), i=1,2.\]
Given a subset $Y \subset B_\Lambda$, define
\[ S(Y) = \{ \chi \in \Short(\Lambda) \; | \; \L \chi, y \R = |y|, \; \forall \, y \in Y \}\]
and \[ S_i(Y) = \{ \chi \in \Short(Z_i) \; | \; \L \chi, \iota_i(y) \R = |y|, \; \forall \, y \in Y \}, \quad i=1,2. \]
In particular, $S(\varnothing) = \Short(\Lambda)$.  Note, crucially, that since $r_i$ surjects, we have
\begin{equation}\label{e: S and r}
S(Y) = \{ r_i(\chi) \; | \; \chi \in S_i(Y) \}, \quad i=1,2.
\end{equation}
On the other hand, \eqref{e: 0/1} implies that
\[ S_1(Y) = \{ \chi \in \Short(Z_1) \; | \; \bigcup_{y \in Y} \supp(\iota_1(y)) \subset \supp^+(\chi) \}. \]
Given an element $x \in B_\Lambda$, define
\[ M(x,Y) = \max\{ \L \chi, x \R \; | \; \chi \in S(Y) \}, \quad  m(x,Y) = \min\{ \L \chi, x \R \; | \; \chi \in S(Y) \}, \]
and
\[ D(x,Y) = \frac{1}{2}(M(x,Y) - m(x,Y)). \]
Hence $M(x,Y)$ is attained by $r_1(\chi_0)$ and $m(x,Y)$ is attained by $r_1(\chi_Y)$, where $\chi_0, \chi_Y \in \Short(Z_1)$ are defined by
\[ \chi_0 = \sum_{e \in B_1} e, \quad \supp^+(\chi_Y) = \bigcup_{y \in Y} \supp(\iota_1(y)).\]
It follows that
\begin{equation}\label{e: D1}
D(x,Y) = | \supp(\iota_1(x)) - \bigcup_{y \in Y} \supp(\iota_1(y)) |.
\end{equation}

\begin{lem}\label{l: 0/1 prep}

We have
\[|  \L \iota_2(x), f \R | \in \{0,1\}, \quad \forall x \in B_\Lambda, f \in B_2. \]
In particular,
\begin{equation}\label{e: supps}
|x| = |\supp(\iota_1(x))| = |\supp(\iota_2(x))|, \quad \forall x \in B_\Lambda.
\end{equation}

\end{lem}

\begin{proof}

Choose $x \in B_\Lambda$.  By \eqref{e: D1} we obtain
\begin{equation}\label{e: M}
D(x,\varnothing) = |\supp(\iota_1(x))| = |x| = |\iota_2(x)| = \sum_{f \in B_2} \L \iota_2(x), f \R^2 .
\end{equation}
On the other hand, $M(x,\varnothing)$ and $m(x,\varnothing)$ are attained by $r_2(\chi_M)$ and $r_2(\chi_m)$, respectively, where $\chi_M, \chi_m \in \Short(Z_2)$ are defined by
\[ \supp^+(\chi_M) = \supp^+(\iota_2(x)), \quad \supp^+(\chi_M) = \supp^-(\iota_2(x)).\]
Hence
\[ D(x,\varnothing) = \sum_{f \in B_2} |\L \iota_2(x), f \R|. \]
Comparing with \eqref{e: M} and the inequality $| \L \iota_2(x), f \R| \leq \L \iota_2(x),f \R^2$ yields
\[ | \L \iota_2(x), f \R| = \L \iota_2(x),f \R^2, \quad \forall f \in B_2,\]
from which the statement of the Lemma follows.

\end{proof}

\begin{lem}\label{l: D2}

We have
\[ \L x, y \R =| \supp(\iota_1(x)) \cap \supp(\iota_1(y)) | = | \supp(\iota_2(x)) \cap \supp(\iota_2(y)) |, \quad \forall x,y \in B_\Lambda. \]
\end{lem}

\begin{proof}

Choose $x,y \in B_\Lambda$. We obtain
\[ S_2(\{y\}) = \{ \chi \in \Short(Z_2) \; | \; \supp^\pm(y) \subset \supp^\pm(\chi) \}. \]
It follows that $M(x,\{y\})$ and $m(x,\{y\})$ are attained by $r_2(\chi_M)$ and $r_2(\chi_m)$, respectively, for the elements $\chi_M, \chi_m \in \Short(Z_2)$ defined by
\[ \supp^+(\chi_M) = \supp^+(\iota_2(y)) \cup \supp^+(\iota_2(x)) - \supp^-(\iota_2(y)), \]
\[ \supp^+(\chi_m) = \supp^+(\iota_2(y)) \cup \supp^-(\iota_2(x)) - \supp^-(\iota_2(y)). \]
We obtain
\begin{equation}\label{e: M diff 2}
D(x,\{y\}) = |\supp(\iota_2(x)) - \supp(\iota_2(y))|.
\end{equation}
The first equality in the Lemma follows from \eqref{e: 0/1}, while the second now follows on combination of \eqref{e: D1}, \eqref{e: supps}, and \eqref{e: M diff 2}.

\end{proof}

\begin{lem}\label{l: 0/1}

There exists an orthonormal basis $B'_2$ for $Z_2$ such that 
\begin{equation}\label{e: 0/1 2}
\L \iota_2(x), f \R \in \{0,1\}, \quad \forall x \in B_\Lambda, f \in B'_2.
\end{equation}

\end{lem}

\begin{proof}

Choose a pair of elements $x, y \in B_\Lambda$.  The pairing $\L \iota_2(x), \iota_2(y) \R$ is a sum of terms
\begin{equation}\label{e: iota terms}
\L \iota_2(x),f \R \cdot \L \iota_2(y),f \R, \quad f \in \supp(\iota_2(x)) \cap \supp(\iota_2(y)),
\end{equation}
each of which is $\pm 1$.  On the other hand, we have
\[ \L \iota_2(x), \iota_2(y) \R = \L x, y \R =  |\supp(\iota_2(x)) \cap \supp(\iota_2(y))| \]
by Lemma \ref{l: D2}.  It follows that each term \eqref{e: iota terms} is $+1$, so
\begin{equation}\label{e: iota pairing}
\L \iota_2(x),f \R = \L \iota_2(y),f \R, \quad \forall \, f \in \supp(\iota_2(x)) \cap \supp(\iota_2(y)).
\end{equation}
For given a fixed $f \in B_2$, it follows from \eqref{e: iota pairing} that either $\L x, f \R \geq 0$,  $\forall x \in B_\Lambda$, or else $\L x, f \R \leq 0$, $\forall x \in B_\Lambda$.  In the first case (which includes the possibility that $\L f, x \R = 0$, $\forall x \in B_\Lambda$), we declare $f \in B'_2$; otherwise, $-f \in B'_2$.  The resulting orthonormal basis $B'_2$ clearly fulfills \eqref{e: 0/1 2}.

\end{proof}

\begin{proof}[Proof of Proposition \ref{p: rigid}]

Using the basis $B'_2$ of Lemma \ref{l: 0/1}, we obtain
\[ S_2(Y) = \{ \chi \in \Short(Z_2) \; | \; \bigcup_{y \in Y} \supp(\iota_2(y)) \subset \supp^+(\chi) \}. \]
Just as \eqref{e: D1} follows from \eqref{e: 0/1}, it follows from \eqref{e: 0/1 2} that
\begin{equation}\label{e: D}
D(x,Y) = | \supp(\iota_i(x)) - \bigcup_{y \in Y} \supp(\iota_i(y)) |, \quad i=1,2.
\end{equation}
Now apply inclusion-exclusion to \eqref{e: D} to obtain, for all partitions $B_\Lambda = X \cup Y$ and $z \in X$,
\begin{eqnarray*}\label{e: inc exc}
| \bigcap_{x \in X}  \supp(\iota_1(x)) - \bigcup_{y \in Y} \supp(\iota_1(y)) | 
&=& \sum_{X' \subset X - z} (-1)^{|X'|} |\supp(\iota_1(z)) - \bigcup_{y \in X' \cup Y} \supp(\iota_1(y))| \\
&=& \sum_{X' \subset X - z} (-1)^{|X'|} D(z,X' \cup Y) \\
&=& \sum_{X' \subset X - z} (-1)^{|X'|} |\supp(\iota_2(z)) - \bigcup_{y \in X' \cup Y} \supp(\iota_2(y))| \\
&=& | \bigcap_{x \in X} \supp(\iota_2(x)) - \bigcup_{y \in Y} \supp(\iota_2(y)) |.
\end{eqnarray*}
Thus, we can set each pair of {\em atoms} into one-to-one correspondence:
\[ \iota_{X,Y}: \bigcap_{x \in X}  \supp(\iota_1(x)) - \bigcup_{y \in Y} \supp(\iota_1(y)) \overset{\sim}{\to}
\bigcap_{x \in X}  \supp(x) - \bigcup_{y \in Y} \supp(y), \]
for all partitions $B_\Lambda = X \cup Y$.  By \eqref{e: full support}, these atoms partition the sets $B_1, B_2$, and piecing together all the various $\iota_{X,Y}$ yields a bijection
\[ \iota: B_1 \overset{\sim}{\to} \bigcup_{x \in B_\Lambda} \supp(\iota_2(x)) \subset B'_2\]
with the property that
\[ \{ \iota(e) \; | \; e \in \supp(\iota_1(x)) \} = \supp(\iota_2(x)), \quad \forall \, x \in B_\Lambda.\]
Extend $\iota$ by linearity to a map $\iota: Z_1 \to Z_2$.  It is clear that $\iota_2 = \iota \circ \iota_1$ since this relation holds for the basis $B_\Lambda$, and this establishes the Proposition.

\end{proof}


\section{Graphs.}


\subsection{The cut lattice and the flow lattice.}\label{ss: cuts and flows}

(cf. \cite{bacheretal:lattices}, \cite[Ch.14]{godsilroyle:book})  Let $G = (V,E)$ denote a finite, loopless, undirected graph with vertex set $V = \{v_1,\dots,v_m\}$ and edge set $E = \{e_1,\dots,e_n\}$, possibly with parallel edges.  Fix an arbitrary orientation $\O_0$ of $G$.  Doing so endows $G$ with the structure of a one-dimensional CW-complex.  Thus, we obtain a short cellular chain complex 
\[ 0 \to C_1(G;\Q) \overset{\del}{\to} C_0(G;\Q) \to 0, \]
where $\del(e) = v-w$ for an edge $e$ oriented from one endpoint $w$ to another $v$.  We equip $C_1(G;\Q)$ and $C_0(G;\Q)$ with inner products by declaring that $E$ and $V$ form orthonormal bases for the respective chain groups.  Doing so enables us to express the adjoint mapping
\[ \del^* : C_0(G;\Q) \to C_1(G;\Q)\]
by the formula
\[\del^*(v) = \sum_{e \in E} \L \del(e),v \R \cdot e.\]
The splitting
\[ \im(\del^*) \oplus \ker(\del) = C_1(G;\Q)\]
gives rise to a pair of sublattices 
\[ \cC(G) := \im(\del^*) \cap C_1(G;\Z) \quad \text{and} \quad \cF(G) := \ker(\del) \cap C_1(G;\Z)\]  
inside $C_1(G;\Z) \cong \Z^n$.  These are the {\em cut lattice} and {\em flow lattice} of $G$, respectively.  Observe that altering the choice of orientation $\O_0$ preserves the isomorphism types of $\cC(G)$ and $\cF(G)$.


\subsection{Bases.}\label{ss: bases}

We recall the standard construction of a pair of bases for $\cC(G)$ and $\cF(G)$ out of a maximal spanning forest $F$ and orientation $\O$ of $G$.  Select an edge $e_i \in E(G)$.  If $e_i \in E(F)$, then the graph $F \setminus e_i$ contains a pair of connected components $K_1$ and $K_2$ with the property that $e_i$ directs from an endpoint in $K_1$ to an endpoint in $K_2$ in $\O$.  The set of edges between $K_1$ and $K_2$ forms the {\em fundamental cut} $\cut(F,e_i)$.  We define the {\em cut orientation} on $\cut(F,e_i)$ by directing each edge out of its endpoint in $K_1$.  Define
\[ x_i = \sum_{v \in V(K_1)} \del^*(v) = \sum_{e_j \in \cut(F,e_i)} \epsilon_j \cdot e_j \in \cC(G),\]
where $\epsilon_j = \pm 1$ according to whether the orientations on $e_j$ in $\cut(F,e_i)$ and $\O$ agree or differ.  If instead $e_i \notin E(F)$, then there exists a unique {\em fundamental cycle} $\cyc(F,e_i)$ in $F \cup e_i$.  We define the {\em cycle orientation} on $\cyc(F,e_i)$ by orienting its edges cyclically, keeping the orientation on $e_i$ from $\O$.   Define
\[ x_i = \sum_{e_j \in \cyc(F,e_i)} \epsilon_j \cdot e_j \in \cF(G),\]
where $\epsilon_j = \pm 1$ according to whether the orientations on $e_j$ in $\cyc(F,e_i)$ and $\O$ agree or differ.  Define a pair of sets
\[ B_\cC = \{ x_i \; | \; e_i \in E(F) \} \quad \text{and} \quad B_\cF = \{ x_i \; | \; e_i \in E(G \setminus F) \}. \]

\begin{prop}\label{p: primitive}

The cut lattice $\cC(G)$ and flow lattice $\cF(G)$ are complementary, primitive sublattices of $C_1(G;\Z)$ with bases $B_\cC$ and $B_\cF$, respectively.

\end{prop}

\begin{proof}

Let $\cC'(G) \subset \cC(G)$ and $\cF'(G) \subset \cF(G)$ denote the spans of $B_\cC$ and $B_\cF$, respectively.  Observe that if $e_i,e_j \in E(F)$, then $\L x_i, e_j \R = \delta_{ij}$, while if $e_i,e_j \in E(G \setminus F)$, then $\L x_i, e_j \R = \delta_{ij}$.  It follows at once that $B_\cC$ and $B_\cF$ are bases for $\cC'(G)$ and $\cF'(G)$, and that $E(F)$ and $E(G \setminus F)$ evaluate on $\cC'(G)$ and $\cF'(G)$ precisely as the dual bases $B_\cC^*$ and $B_\cF^*$, respectively. Thus, $\cC'(G)$ and $\cF'(G)$ are primitive sublattices of $C_1(G;\Z)$.

Now,
\[ m = |B_\cC| + |B_\cF| = \rk (\cC'(G)) + \rk(\cF'(G)) \leq \rk (\cC(G)) + \rk (\cF(G))\leq m,\]
where the last inequality follows since $\cC(G)$ and $\cF(G)$ are orthogonal.  It follows that $\cC(G)$ and $\cF(G)$ are complementary, and furthermore that $B_\cC$ and $B_\cF$ are bases for the vector spaces $\cC(G) \otimes \Q = \im(\del^*)$ and $\cF(G) \otimes \Q = \ker(\del)$, respectively.  Thus, any $x \in \cC(G)$ has an expression
\[ x = \sum_{e_i \in E(F)} q_i \cdot x_i, \quad q_i \in \Q.\]
However, since $q_i = \L x, e_i \R \in \Z$, we must in fact have $x \in \cC'(G)$.  Hence $\cC'(G) = \cC(G)$, and similarly $\cF'(G) = \cF(G)$.  The statement of the Lemma now follows.

\end{proof}

The following Lemma ensures a particularly nice choice of spanning forest and orientation (cf. \eqref{e: 0/1} in Proposition \ref{p: rigid}).

\begin{lem}\label{l: special basis}

There exists a maximal spanning forest $F$ and an orientation $\O_1$ such that
\[ \L x_i, e_j \R \in \{ 0, 1 \}, \quad \forall x_i \in B_\cC, \, e_j \in E(G), \]
and an orientation $\O_2$ such that
\[ \L x_i, e_j \R \in \{ 0, 1 \}, \quad \forall x_i \in B_\cF, \, e_j \in E(G). \]

\end{lem}

\begin{proof}

A {\em root set} $R$ in a graph is a subset of its vertices, one in each connected component.  Let $R_1$ be a root set of $G_1 = G$.  Having defined $R_i$ and $G_i$, let $G_{i+1} = G_i - R_i$, choose a root set $R_{i+1}$ in $G_{i+1}$ with the property that each vertex $v_{i+1} \in R_{i+1}$ has a (unique) neighbor $v_i \in R_i$, and distinguish a single edge $e = (v_i,v_{i+1})$.  Let $F$ be the subgraph of $G$ consisting of all such edges.

By induction on $i$, no vertex in $R_i$ is contained in a cycle in $F$, hence $F$ is a forest.  By reverse induction on $i$, $R_i$ is a root set for the subgraph of $F$ induced on $V(G_i)$, hence ($i=1$) $F$ is maximal.  Given an edge $e \in E(G)$, write $e = (v_i,v_j)$, where $v_i \in R_i$, $v_j \in R_j$, and $i < j$ ($i \ne j$ since each $R_i$ is an independent set).  We obtain an orientation $\O_1$ of $G$ by orienting each edge $e$ from $v_i$ to $v_j$, and another orientation $\O_2$ by reversing the orientation on each edge in $E(G \setminus F)$.

Observe that for all $e \in E(F)$, every edge in $\cut(F,e)$ directs the same way in the cut orientation and $\O_1$.  Similarly, for all $e \in E(G \setminus F)$, every edge in $\cyc(F,e)$ directs the same way in the cycle orientation and $\O_2$.  The statement of the Lemma now follows for this choice of $F$, $\O_1$, and $\O_2$.

\end{proof}


\subsection{Short characteristic covectors.}

Fixing an orientation $\O_0$, there exists a 1-1 correspondence
\[ \Short(C_1(G;\Z)) \leftrightarrow  \{ \text{orientations } \O \text{ of } G\}, \]
where $\chi_\O \leftrightarrow \O$ is determined by specifying that $\L \chi_\O, e_i \R = 1$ if $e_i$ gets the same orientation in both $\O_0$ and $\O$ and $-1$ otherwise.  The value $\L \chi_\O, \del^*(v) \R$ is thus {\em minus} the signed degree of $v$ in $\O$: it equals the number of edges in $D$ directed into $v$ minus the number directed out of it, i.e.
\[ \L \chi_\O, \del^*(v) \R = -\deg_\O(v) = \deg_\O^{\, \text{in}}(v) - \deg_\O^{\, \text{out}}(v).\]
We denote the restriction of $\chi_\O$ to $\Short(\cC(G))$ by the same symbol and call it an {\em orientation covector} for $\cC(G)$.

\begin{prop}\label{p: short graph}

The set $\Short(\cC(G))$ consists of precisely the orientation covectors for $\cC(G)$.
\end{prop}

\begin{proof}

This follows in essence from a result of Hakimi \cite[Thm.4]{hakimi}; we follow the elegant treatment of Schrijver \cite[Thm.61.1\&Cor.61.1a]{schrijver:book}.  Thus, suppose that $\chi \in \Short(\cC(G))$.  From $|\chi \pm 2 \sum_{v \in T} \del^*(v) | \geq |\chi|$ we obtain
\begin{equation}\label{e: ineq 2}
|\L \chi,  \sum_{v \in T} \del^*(v) \R| \leq | \sum_{v \in T} \del^*(v)|, \quad  \forall \, T \subset V.
\end{equation}
Define a function $l : V \to \Z_{\geq 0}$ by
\[ l(v) = {1 \over 2}(\deg(v) - \L \chi, \del^*(v) \R), \]
and extend $l$ to subsets of $V$ by declaring $l(T) = \sum_{v \in T} l(v)$.  Observe that $l$ satisfies two key properties:
\begin{equation}\label{e: Hall 1}
l(V) = |E|, \text{ and}
\end{equation}
\begin{equation}\label{e: Hall 2}
l(T) \leq e(T), \quad \forall \, T \subset V,
\end{equation}
where $e(T)$ denotes the number of edges with at least one endpoint in $T$.

We seek an orientation $\O$ of $G$ with the property that $\deg^\text{\,in}_\O(v) = l(v)$ for all $v \in V$; then $\L \chi, \del^*(v) \R = -\deg_\O(v)$, so $\chi = \chi_\O$ is an orientation covector.  To produce $\O$, construct a bipartite graph $B$ with two partite classes: $V'$, which contains $l(v)$ copies of $v$ for each $v \in V$; and $E$, the edge set of $G$.  The edge set of $B$ consists of pairs $(v,e)$, where $v \in V'$ denotes a copy of an endpoint of $e \in E$.  Properties \eqref{e: Hall 1} and \eqref{e: Hall 2} ensure that for every subset $T' \subset V'$, there exist at least $|T'|$ elements of $E$ with a neighbor in $V'$.  Thus, Hall's matching theorem implies that $B$ contains a perfect matching $\mathcal{M}$ \cite[Thm.16.7]{schrijver:book}.  Directing each $e \in E$ to the endpoint to which it gets matched in $\mathcal{M}$ produces the desired orientation $\O$.

\end{proof}

\begin{cor}\label{c: cut and cycle}

The restriction maps
\[ \Short(C_1(G;\Z)) \to \Short(\cC(G)), \, \Short(\cF(G))\]
surject, and the inclusion $\cC(G) \oplus \cF(G) \subset C_1(G;\Z)$ induces a natural isomorphism
\[ \varphi: (C(\cC(G)),d_\cC) \overset{\sim}{\to} (C(\cF(G)),-d_\cF). \]

\end{cor}

\begin{proof}

This follows immediately on combination of Propositions \ref{p: gluing}, \ref{p: primitive}, and \ref{p: short graph}. 

\end{proof}


\subsection{Whitney's theorem.}

Now suppose that $G$ is connected (by convention, the empty graph is connected).  A {\em cut-edge} $e \in E(G)$ is one such that $G - e$ is disconnected, and a {\em cut-vertex} $v \in V(G)$ is one such that $G - v$ is disconnected.  The graph $G$ is {\em 2-edge-connected} if it is does not contain a cut-edge and {\em 2-connected} if it does not contain a cut-vertex.  It is straightforward to show that $G$ is 2-edge-connected iff every edge is contained in some cycle, and 2-connected iff every {\em pair} of distinct edges is contained in some cycle.  Thus, a 2-connected graph is 2-edge-connected, and the graph with one vertex and no edge is 2-edge-connected.  A {\em 2-isomorphism} between a pair of graphs is a cycle-preserving bijection between their edge sets.

A special instance of 2-isomorphism arises as follows.  Let $G_1, G_2$ denote a pair of disjoint graphs, and distinguish a pair of distinct vertices $v_i,w_i \in V(G_i)$, $i=1,2$.  Form a graph $G$ by identifying the vertices $v_1, v_2$ into a vertex $v$ and vertices $w_1,w_2$ into a vertex $w$; and similarly, form a graph $G'$ by identifying the vertices $v_1, w_2$ into a vertex $v'$ and vertices $w_1,v_2$ into a vertex $w'$.  We say that $G$ and $G'$ are related by a {\em switch}.  The switch is {\em special} if one of $v_i, w_i$ is an isolated vertex in $G_i$ for some $i$.  In this case, one of $v,w$ is a cut-vertex in $G$ and one of $v',w'$ is a cut-vertex in $G'$.  It is clear that identifying $E(G_i) \subset E(G)$ with $E(G_i) \subset E(G')$, $i=1,2$, defines a 2-isomorphism between $G$ and $G'$.

Conversely, we have the following important fact.

\begin{thm}[Whitney \cite{whitney:switch}]\label{t: switch}

A pair of 2-connected graphs are 2-isomorphic iff they are related by a sequence of switches. \qed

\end{thm}

\noindent Truemper gave a short, simple proof of Theorem \ref{t: switch} \cite{truemper:switch}.

We now develop a straightforward generalization of Theorem \ref{t: switch} to the case of an arbitrary connected graph $G$.  A {\em block} $B \subset G$ is a maximal 2-connected subgraph of $G$.  In particular, the cut-edges of $G$ constitute its 1-edge blocks, and the cut-vertices of $G$ are the vertices of intersection between distinct blocks of $G$.  Let $T(G)$ denote the set of edges contained in some cycle in $G$; thus, $e \in T(G)$ iff $e$ is not a cut-edge.  A {\em 2-isomorphism} between a pair of connected graphs $G,G'$ is a cycle-preserving bijection between $T(G)$ and $T(G')$.

Given a cut-edge $e \in E(G)$, we contract it to obtain a new graph $G / e$.  We say that $G/e$ is obtained from $G$ by {\em cut-edge contraction}, and conversely that $G$ is obtained from $G/e$ by {\em cut-edge expansion}.  It is clear that both cut-edge contraction and expansion preserve the 2-isomorphism type of a graph.

With these definitions in place, we state the desired generalization of Theorem \ref{t: switch}.

\begin{prop}\label{p: switch}

A pair of connected graphs are 2-isomorphic iff they are related by a sequence of switches and cut-edge contractions and expansions.  Furthermore, only switches are necessary if the graphs are 2-edge-connected.

\end{prop}

\begin{proof}

For the first part, we just need to establish the forward implication.  Write $H \approx H'$ if $H$ is related to $H'$ by a sequence of switches and cut-edge contractions and expansions..  Clearly, $\approx$ defines an equivalence relation on graphs.  Now, suppose that $G$ and $G'$ are a pair of 2-isomorphic, connected graphs.  In each graph, contract all the cut-edges and perform a sequence of special switches so that there is a vertex in common to all remaining blocks (it will be the unique cut-vertex if there are multiple blocks).  The resulting graphs $G_0 \approx G$ and $G_0' \approx G'$ are 2-isomorphic by some mapping $\varphi$.  Put an equivalence relation $\sim$ on $E(H)$ by declaring $e \sim f$ if $e=f$ or $e$ and $f$ belong to some cycle.  Thus, the edge sets of blocks of $H$ are precisely the equivalence classes under $\sim$.  Since $\varphi$ clearly preserves $\sim$, it follows that $\varphi$ pairs the blocks of $G_0$ and $G'_0$, and furthermore defines a 2-isomorphism between each such pair $(B_0,B'_0)$.  By Theorem \ref{t: switch}, it follows that $B_0$ and $B'_0$ are related by a sequence of switches.  Each switch in $B_0$ extends to a switch in $G_0$, the composition of which results in a graph $\overline{G} \approx G_0$ whose blocks are isomorphic in pairs with those of $G'_0$.  A sequence of special switches now transforms $\overline{G}$ into $G'_0$.  Thus, $G \approx G_0 \approx \overline{G} \approx G'_0 \approx G'$, as desired.  Lastly, if $G$ and $G'$ are 2-edge-connected, then $G = G_0$ and $G'=G'_0$, so only switches are necessary to establish $G \approx G'$.

\end{proof}


\subsection{Graph lattices with the same $d$-invariant.}

For a pair of lattices $\Lambda_1, \Lambda_2$, write $\Lambda_1 \simeq \Lambda_2$ if $\Lambda_1 \oplus \Z^k \cong \Lambda_2$ or $\Lambda_1  \cong \Lambda_2 \oplus \Z^k$ for some $k$.  The following Proposition and its proof are essentially due to Bacher, et al. \cite[Prop.5]{bacheretal:lattices}.

\begin{prop}\label{p: main prep}

If $G$ and $G'$ are 2-isomorphic, then $\cF(G) \cong \cF(G')$ and $\cC(G) \simeq \cC(G')$.

\end{prop}

\begin{proof}

By Proposition \ref{p: switch}, it suffices to establish the statement of the Proposition under the assumption that $G'$ is obtained from $G$ by a switch or a cut-edge contraction.  

Suppose first that $G' = G/e$ for a cut-edge $e \in E(G)$.  An orientation on $G$ induces one on $G'$ and identifies $e$ with a basis element in $C_1(G;\Z)$.  We obtain a natural isomorphism between $e^\perp \subset C_1(G;\Z)$ and $C_1(G';\Z)$.  This isomorphism clearly carries $\cF(G) \subset e^\perp$ onto $\cF(G')$ and $\cF(G)^\perp \cap e^\perp$ onto $\cF(G')^\perp = \cC(G')$.  Note as well that $\cC(G) = \cF(G)^\perp =( \cF(G)^\perp \cap e^\perp) \oplus (e) \cong \cC(G') \oplus \Z$. Thus, $\cF(G) \cong \cF(G')$ and $\cC(G) \simeq \cC(G')$.

Next, suppose that $G$ and $G'$ are related by a switch.  Choose an orientation $\O_i$ of $G_i$, $i=1,2$, and let $\O, \O'$ denote the induced orientations of $G, G'$.  Doing so leads to natural isomorphisms
\[ C_1(G_1;\Z) \oplus C_1(G_2;\Z) \overset{\sim}{\to} C_1(G;\Z), \quad (x,y) \mapsto x+y \]
and
\[ C_1(G_1;\Z) \oplus C_1(G_2;\Z) \overset{\sim}{\to} C_1(G';\Z), \quad (x,y) \mapsto x-y .\]
Define
\[ \widetilde{\cF}(G_i) = \{ x \in C_1(G_i;\Z) \; | \;  \L \del x,u \R = 0, \forall u \ne v_i,w_i\}, \quad  i=1,2.\]
Observe that $\im(\del)$ is contained in the kernel of the augmentation map $C_0 \to \Z$ defined by sending all vertices to $1$.  It follows that $x \in \widetilde{\cF}(G_i)$ satisfies
\begin{equation}\label{e: aug}
\L \del x, v_i \R + \L \del x, w_i \R = 0, \quad i=1,2.
\end{equation}
Thus, the preceding isomorphisms restrict to isomorphisms
\[ \{ (x,y) \in \widetilde{\cF}(G_1) \oplus\widetilde{\cF}(G_2) \; | \; \L \del x , v_1 \R + \L \del y, v_2 \R = 0 \} \overset{\sim}{\to} \cF(G)\]
and
\[ \{ (x,y) \in \widetilde{\cF}(G_1) \oplus \widetilde{\cF}(G_2) \; | \; \L \del x, w_1 \R - \L \del y, v_2 \R = 0 \} \overset{\sim}{\to} \cF(G'),\]
respectively.  By \eqref{e: aug}, the two sublattices of $\widetilde{\cF}(G_1) \oplus \widetilde{\cF}(G_2)$ appearing here coincide.  Thus, $\cF(G) \cong \cF(G')$.  Since this isomorphism is induced  by the isomorphism $C_1(G;\Z) \cong C_1(G';\Z)$, it follows that their orthogonal complements are isomorphic as well: $\cC(G) \cong \cC(G')$.  This completes the proof of the Proposition.

\end{proof}

We now state our main combinatorial result, an extension of Theorem \ref{t: main graph preview}.

\begin{thm}\label{t: main graph}

The following assertions are equivalent for a pair of connected graphs $G, G'$:

\begin{enumerate}

\item $G$ and $G'$ are 2-isomorphic;

\item $\cF(G) \cong \cF(G')$;

\item $\cC(G) \simeq \cC(G')$;

\item the $d$-invariants of $\cF(G)$ and $\cF(G')$ are isomorphic; and

\item the $d$-invariants of $\cC(G)$ and $\cC(G')$ are isomorphic.

\end{enumerate}

\end{thm}

\begin{proof}

(1)$\implies$(2),(3) follow from Proposition \ref{p: main prep}; (2)$\implies$(4) and (3)$\implies$(5) are clear (using additivity of the $d$-invariant in the second case); and (4)$\iff$(5) follows from Corollary \ref{c: cut and cycle}.  We proceed to establish (5) $\implies$(1), from which the Theorem follows.

Note that contracting all the cut-edges in a connected graph results in a 2-edge-connected graph  with the same 2-isomorphism type.  Hence it suffices to establish (5)$\implies$(1) under the assumption that both graphs are 2-edge-connected.  Thus, suppose that $(C(\cC(G)),d_\cC) \cong (C(\cC(G')),d'_\cC)$ for a pair of 2-edge-connected graphs $G, G'$. By Corollary \ref{c: cut and cycle}, there exists a natural isomorphism $(C(\cC(G')),d'_\cC) \overset{\sim}{\to} (C(\cF(G')),-d'_\cF)$.  Consequently, we obtain an isomorphism
\[ \varphi: (C(\cC(G)),d_\cC) \overset{\sim}{\to} (C(\cF(G')),-d'_\cF).\]
Proposition \ref{p: gluing} implies that the glue lattice
\[ Z_2 := \cC(G) \oplus_\varphi \cF(G') \]
admits an orthonormal basis, and furthermore that the restriction maps
\[ \Short(Z_2) \to \cC(G), \, \cF(G') \]
surject.

Now, let $(\Lambda, Z_1)$ denote either pair $(\cC(G),C_1(G;\Z))$ or $(\cF(G'),C_1(G';\Z))$.  By Corollary \ref{c: cut and cycle}, the restriction map $\Short(Z_1) \to \Lambda$ surjects, and by Lemma \ref{l: special basis}, $Z_1$ admits an orthonormal basis $B_1$ such that \eqref{e: 0/1} holds.  Every edge of $G$ is contained in some cut, and by 2-edge-connectivity, every edge of $G'$ is contained in some cycle.  It follows in either case that \eqref{e: full support} holds.  Thus, Proposition \ref{p: rigid} applies and furnishes embeddings $\iota$, $\iota'$ such that the following diagram commutes:
\begin{equation}\label{e: diagram}
\xymatrix {
C_1(G;\Z) \; \ar@{^{(}->}[rr]^-\iota & & Z_2 & & \; C_1(G';\Z) \ar@{_{(}->}[ll]_-{\iota'} \cr
& \cC(G) \ar@{_{(}->}[ul]^-{\iota^{\phantom{o}}_1} \ar@{^{(}->}[ur]_-{\iota^{\phantom{o}}_2} & & \cF(G') \ar@{_{(}->}[ul]^-{\iota'_2} \ar@{^{(}->}[ur]_-{\iota'_1} &
}
\end{equation}

Switching the roles of $G$ and $G'$, we can repeat the same construction with respect to the glue lattice
\[ Z'_2 := \cC(G') \oplus_\psi \cF(G),\]
using an isomorphism
\[ \psi: (C(\cC(G')),d'_\cC) \overset{\sim}{\to} (C(\cF(G)),-d_\cF).\]
We obtain two more embeddings from Proposition \ref{p: rigid}, leading to a total of four inequalities
\begin{equation}\label{e: four ineq}
\rk(C_1(G;\Z)), \, \rk(C_1(G';\Z)) \leq \rk(Z_2), \, \rk(Z'_2).
\end{equation}
On the other hand, 
\begin{eqnarray*}
\rk(C_1(G;\Z)) + \rk(C_1(G';\Z)) &=& \rk(\cC(G)) + \rk(\cF(G)) + \rk(\cC(G')) + \rk(\cF(G')) \\
&=& \rk(Z_2) + \rk(Z'_2).
\end{eqnarray*}
Hence each inequality \eqref{e: four ineq} is an equality, so the embeddings $\iota$, $\iota'$ are actually isomorphisms.  Thus, we obtain a composite isomorphism
\[ f:= (\iota')^{-1} \circ \iota :  C_1(G;\Z) \overset{\sim}{\to} C_1(G';\Z), \]
and restricting $f$ to the orthonormal bases induces a bijection
\[f_E: E(G) \overset{\sim}{\to} E(G').\]

We claim that $f_E$ is a 2-isomorphism.  First note that \eqref{e: diagram} and Propositions \ref{p: gluing} and \ref{p: primitive} show that $f$ carries $\cF(G)$ isomorphically onto $\cF(G')$.  Now let $C$ denote a cycle in $G$.  With an orientation $\O$ of $G$ fixed and an arbitrary edge $e \in E(C)$ distinguished, we obtain an element $x(C) \in \cF(G) \subset C_1(G;\Z)$ as in $\S$\ref{ss: bases}.  Thus we obtain an element $f(x(C)) \in \cF(G') \subset C_1(G';\Z)$ with the property that
\[ |\L f(x(C)), e' \R| \in \{0,1\}, \quad \forall e' \in C_1(G';\Z), |e'| = 1. \]
With an orientation $\O'$ of $G'$ fixed, it follows that the subgraph $f_E(E(C))$ is an (oriented) Eulerian subgraph of $G'$.  Hence it decomposes into an edge-disjoint union of directed cycles.  Choose one and denote it by $C'$.  By symmetry, $f_E^{-1}$ carries $C'$ onto a non-empty Eulerian subgraph of $C$; but since $C$ is a cycle, it follows that $f_E^{-1}(E(C')) = E(C)$.  Hence $f_E$ carries the cycle $C$ to the cycle $C'$. Since $C$ was arbitrary, it follows that $f_E$ is a 2-isomorphism, as claimed.

\end{proof}


\section{Conway Mutation.}\label{s: mutation}


\subsection{Planar graphs.}

By abuse of terminology, we regard a plane drawing $\Gamma$ of a planar graph $G$ as an embedding in the sphere $S^2 = \bR^2 \cup \{ \infty \}$.

Connectivity properties of $G$ are reflected by the topology of $\Gamma$ in the following way.  Suppose that $\{v, w\}$ is a cut-set in $G$, where $v,w \in V(G)$ need not be distinct.  Then there exists a circle $S^1 \subset S^2$ such that $S^1 \cap \Gamma = \{v,w\}$ and both components of $S^2 - S^1$ contain a vertex of $G$.  Conversely, given such a circle with $S^1 \cap \Gamma = \{v,w\}$, it follows that $\{v,w\}$ is a cut-set in $G$.

Choose either disk bounded by $S^1$ and reglue it by an orientation-reversing homeomorphism that fixes $v$ and $w$.  Doing so results in another plane drawing $\Gamma'$ of $G$, and we say that $\Gamma, \Gamma'$ differ by a {\em flip}.  Conversely, we have the following result.

\begin{prop}[Mohar-Thomassen \cite{mt:book}, Thm.2.6.8\footnote{As remarked on \cite[p.3]{mt:book}, the results of that book are stated for {\em simple} graphs, i.e. those without parallel edges, but most results (including this one) apply, {\em mutatis mutandis}, to graphs with parallel edges.}]\label{p: flip}

Any two plane drawings of a 2-connected planar graph $G$ are related by a sequence of flips and isotopies in the sphere. \qed

\end{prop}

Alternatively, choose either disk bounded by $S^1$ and reglue it by a homeomorphism that exchanges $v$ and $w$.  Doing so results in a plane drawing $\Gamma'$ of a graph $G'$ related to $G$ by a switch. We say that $\Gamma, \Gamma'$ differ by a {\em planar switch}.  The planar switch is {\em positive} or {\em negative} according to whether the homeomorphism preserves or reverses orientation. Note that a positive and negative planar switch differ by composition with a flip.

Lastly, suppose that there exists a pair of disks $D^2_1, D^2_2$ such that $D^2_1 \cap \Gamma = \Gamma \cap D^2_2 = D^2_2 \cap D^2_1 = \{v \}$ for some $v \in V(\Gamma)$.  Exchange $D^2_1$ and $D^2_2$ by a homeomorphism that preserves $v$; doing so results in another plane drawing $\Gamma'$ of $G$, and we say that $\Gamma, \Gamma'$ differ by a {\em swap}.  The swap is {\em positive} or {\em negative} according to whether the homeomorphism preserves or reverses orientation.

Examples of flips, planar switches, and swaps appear in Figures \ref{f: mutation1}-\ref{f: mutation3}.

\begin{lem}\label{l: swap}

Any two plane drawings $\Gamma, \Gamma'$ of a connected planar graph $G$ are related by a sequence of flips, swaps, and isotopies.

\end{lem}

\begin{proof}[Proof sketch.]

If $G$ is 2-connected, then Proposition \ref{p: flip} applies at once, so suppose otherwise and fix a cut-vertex $v \in V(G)$.  Decompose $G$ uniquely into a maximal collection of subgraphs $G_1,\dots,G_k$ that intersect pairwise in $v$.  For each $i$, let $\Gamma_i \subset \Gamma, \Gamma'_i \subset \Gamma'$ denote the induced plane drawings of $G_i$.  Reindexing the subgraphs if necessary, there exists a sequence of disks $D^2_1, \dots, D^2_k \subset S^2$ whose boundaries intersect pairwise in $v$ and such that
\[ \Gamma_i = \Gamma \cap (D^2_i \setminus \text{int} (\bigcup_{j=1}^{i-1} D^2_j)), \quad \forall i.\]
A sequence of at most $k-2$ swaps results in a plane drawing $\overline{\Gamma}$ of $G$ such that there exist disks $\overline{D^2_1},\dots,\overline{D^2_k}$ that intersect pairwise in $v$ and satisfy $\overline{\Gamma}_i = \overline{\Gamma} \cap \overline{D^2_i}$.  Similarly, a sequence of swaps and isotopies transforms $\Gamma'$ into a plane drawing $\overline{\Gamma}'$ of $G$ such that $\overline{\Gamma}'_i = \overline{\Gamma}' \cap \overline{D^2_i}$ for this same collection of disks.  By induction on $|V|$, there exists a sequence of flips, swaps, and isotopies supported in $\overline{D^2_i}$ that transforms $\overline{\Gamma}_i \subset S^2_i := \overline{D^2_i} / \del \overline{D^2_i}$ into $\overline{\Gamma}'_i \subset S^2_i$.  This sequence, together with another sequence of swaps in $S^2$, transforms $\overline{\Gamma}$ into $\overline{\Gamma}'$.  Thus, $\Gamma$ and $\Gamma'$ are related in the desired manner.

\end{proof}

\begin{cor}\label{c: flip and switch}

Let $\Gamma, \Gamma'$ denote plane drawings of a pair of 2-isomorphic, 2-edge-connected planar graphs $G,G'$.  Then $\Gamma$ and $\Gamma'$ are related by a sequence of flips, planar switches, swaps, and isotopies.

\end{cor}

\begin{proof}

Suppose first that $G, G'$ are related by a single switch.  In this case, there clearly exists a planar switch of  $\Gamma$ that results in a plane drawing $\Gamma'_0$ of $G'$.  By Lemma \ref{l: swap}, there exists a sequence of flips, swaps, and isotopies that transforms $\Gamma'_0$ into $\Gamma'$.  The general case of the Corollary now follows from Proposition \ref{p: switch}.

\end{proof}


\subsection{Between diagrams and graphs.}\label{ss: mutation}

\begin{figure}
\centering
\includegraphics[width=1in]{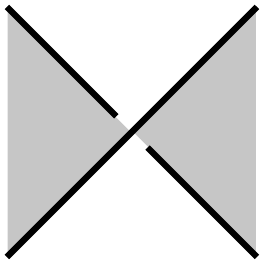}
\caption{Coloring convention for alternating diagrams.}  \label{f: coloringconvention}
\end{figure}

Let $D \subset S^2$ denote a connected, alternating diagram for a link $L$.  Color the regions of $D$ black and white in checkerboard fashion according to the coloring convention displayed in Figure \ref{f: coloringconvention}.  We obtain a planar graph by drawing a vertex in each black region and an edge for each crossing that joins a pair of black regions.  Examples appear in Figures \ref{f: mutation1}-\ref{f: mutation3}.  The result is the {\em Tait graph} $G$ of $D$, equipped with a natural (isotopy class of) plane drawing $\Gamma$.  This process is clearly reversible: given a connected plane drawing $\Gamma$, we obtain from it a connected, alternating link diagram $D$.

For concreteness, write $S^3 = \bR^3 \cup \{ \infty \}$, where $\bR^3$ has coordinates $x,y,z$, and let $S^2 \subset S^3$ denote the $xy$-coordinate plane together with $\infty$.  Suppose that the unit ball $B^3 \subset \bR^3$ meets $L$ in the four points $\{ (\pm 1/ \sqrt{2}, \pm 1 / \sqrt{2}, 0) \}$, and these are all regular points for $D$.  The sphere $\del B^3$ is a {\em Conway sphere} for $L$, and the circle $\del B^3 \cap S^2$ is a {\em Conway circle} for $D$.  More generally, a Conway circle for $D$ refers to any circle $S^1 \subset S^2$ meeting $D$ in four regular points; by a suitable isotopy, we can arrange that $S^1$ arises in the manner just described.  We operate on $B^3 \cap L$ by performing a $180^\circ$ rotation about one of the three coordinate axes.  The result is a link $L'$ and a corresponding diagram $D' \subset S^2$.  We say that the links $L, L'$ differ by a ({\em Conway}) {\em mutation}, and a pair of links are {\em mutants} if they differ by a sequence of isotopies and mutations.  We make similar definitions at the level of diagrams, requiring all isotopies to take place in $S^2$.  Thus, mutant diagrams present mutant links, but the converse does not hold in general.

\begin{lem}\label{l: graph mutation}

Let $D$ denote a connected, alternating diagram and $\Gamma$ the associated plane drawing of its Tait graph.  A mutation of $D$ effects one or two flips, a planar switch, or a swap in $\Gamma$.  Conversely, a flip, planar switch, or swap in $\Gamma$ effects a mutation of $D$.

\end{lem}

\begin{proof}

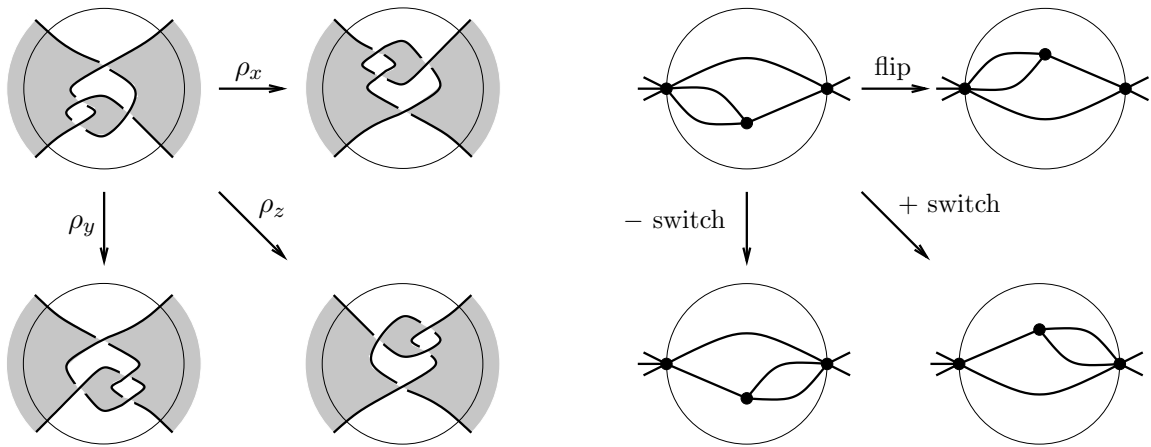
\begin{figure}
\centering
\includegraphics[width=6in]{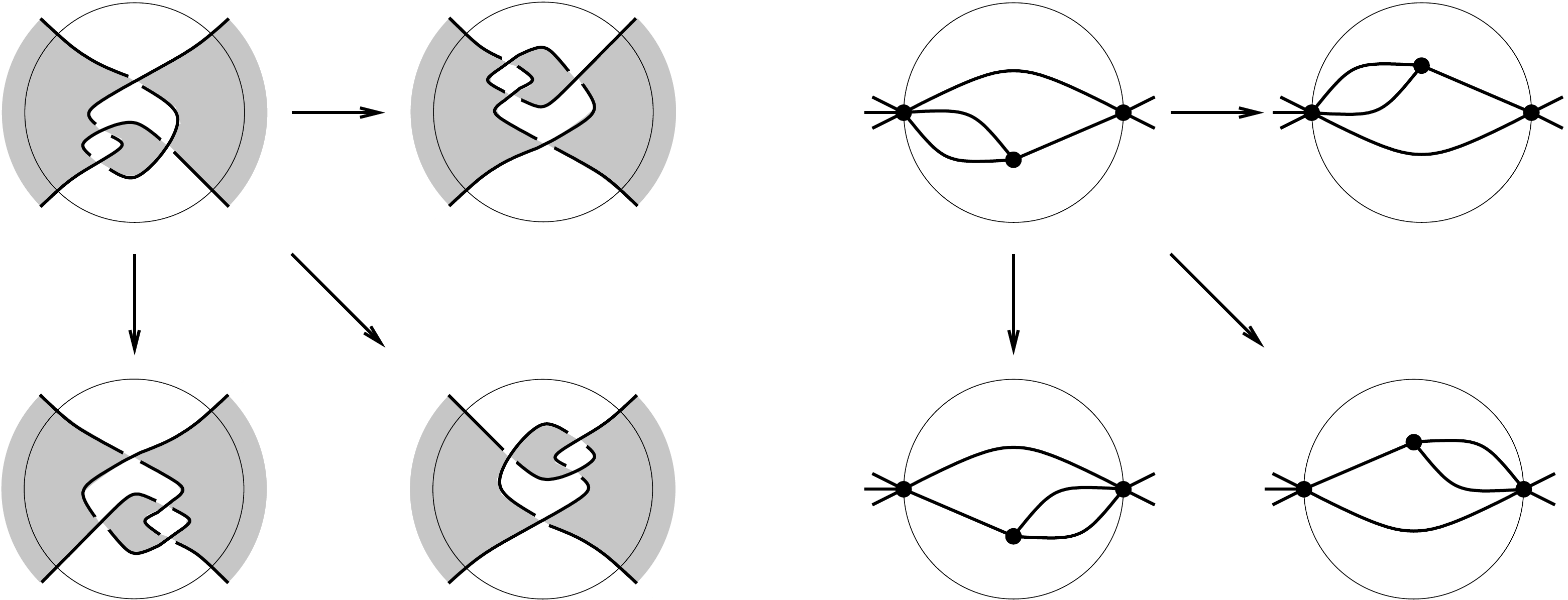}
\put(-347,140){$\rho_x$}
\put(-410,82){$\rho_y$}
\put(-338,88){$\rho_z$}
\put(-105,140){\small flip}
\put(-96,88){\small $+$ switch}
\put(-200,82){\small $-$ switch}
\caption{Between mutations and graph operations (1).}  \label{f: mutation1}
\end{figure}

Draw $D$ and $\Gamma$ simultaneously in $S^2$, and choose a Conway circle for $D$ with respect to which to mutate.  Choose coordinates so that the Conway circle arises in the concrete manner described above, and let $D^2_0$, $D^2_\infty \subset S^2$ denote the two disks that it bounds.  By rotating the diagram $90^\circ$ if necessary, we may assume that the points $(\pm1,0,0)$ lie in black regions of $D$.

For the forward implication, we distinguish two cases, depending on whether or not $(-1,0,0)$ belongs to the same black region as $(1,0,0)$ or not.  Suppose first that it belongs to a different black region.  By an isotopy of  $\Gamma$, we may assume that the points $(\pm1,0,0)$ represent distinct vertices $v,w \in V(\Gamma)$.  Let $\Gamma_1 \subset \Gamma$ denote the subgraph induced on the regions of $D^2_0$  and $\Gamma_2 \subset \Gamma$ the subgraph induced on the regions of $D^2_\infty$.  By inspection, rotation of the unit disk about the $x$-axis corresponds to a flip of $\Gamma_1 \subset \Gamma$; rotation about the $y$-axis corresponds to a negative planar switch; and rotation about the $z$-axis corresponds to a positive planar switch (Figure \ref{f: mutation1}).  This establishes the forward implication of the Lemma in this case.

Suppose instead that $(-1,0,0)$ belongs to the same black region as $(1,0,0)$.  By an isotopy of $D$, and possibly a change of coordinates that exchanges $D^2_0$ and $D^2_\infty$, we may assume that the interval $\{ (t,0,0), -1 \leq t \leq 1\} $ is supported in this black region.  Thus, the diagram meets the unit disk in a split pair of (possibly knotted) strands.  By an isotopy of $\Gamma$, we may assume that $(0,0,0)$ represents a vertex $v \in V(\Gamma)$.  Let $\Gamma_1 \subset \Gamma$ denote the subgraph induced on the regions of $D^2_0 \cap \{y \geq 0\}$ and $\Gamma_2 \subset \Gamma$ the subgraph induced on the regions of $D^2_0 \cap \{y \leq 0\}$.  By inspection, rotation of the unit disk about the $x$-axis corresponds to a negative swap of $\Gamma_1, \Gamma_2 \subset \Gamma$; rotation about the $y$-axis corresponds to flipping both $\Gamma_1$ and $\Gamma_2$; and rotation about the $z$-axis corresponds to positive swap (Figure \ref{f: mutation2}).  This establishes the forward implication of the Lemma in this case.

\begin{figure}
\centering
\includegraphics[width=6in]{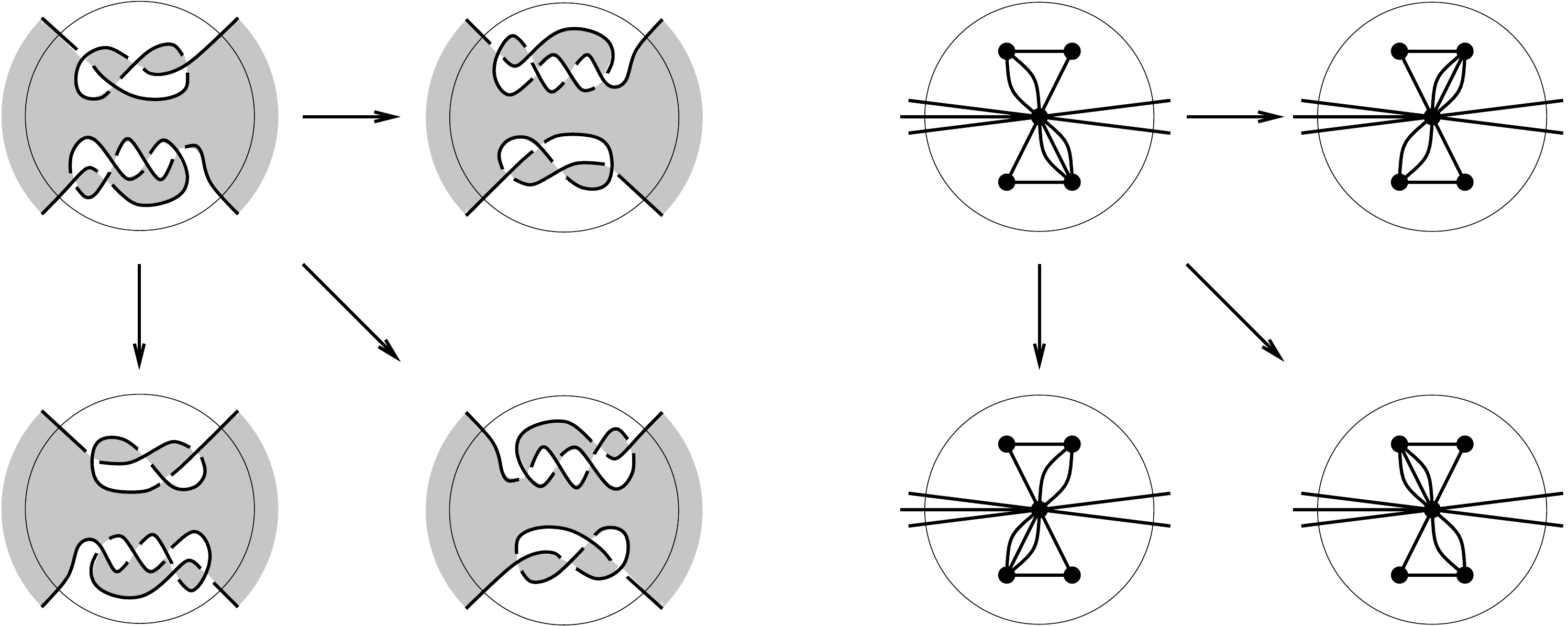}
\put(-343,147){$\rho_x$}
\put(-410,87){$\rho_y$}
\put(-333,90){$\rho_z$}
\put(-108,157){\small $-$ swap}
\put(-91,91){\small $+$ swap}
\put(-169,90){\small two}
\put(-170,80){\small flips}
\caption{Between mutations and graph operations (2).}  \label{f: mutation2}
\end{figure}
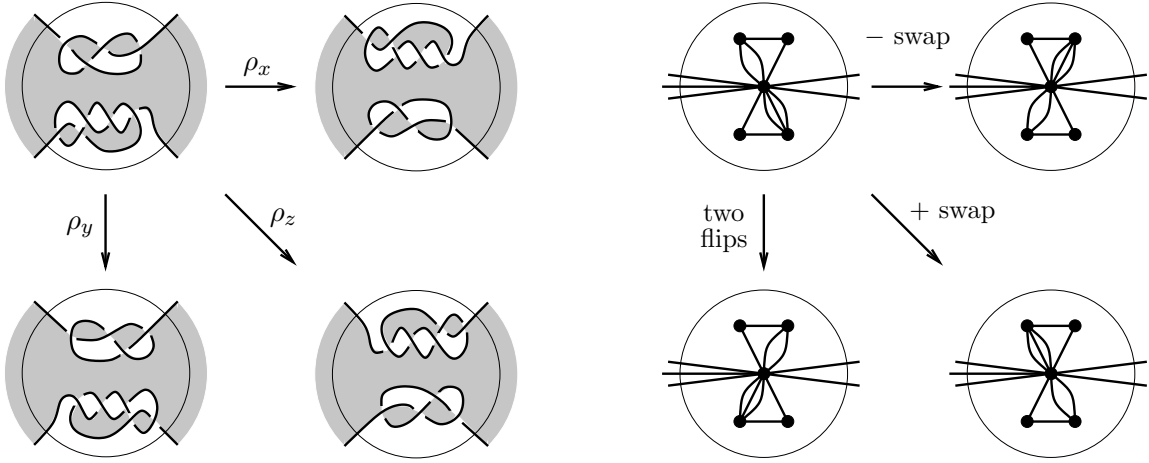

For the reverse direction, we distinguish several cases as well.  First consider the case of a flip or a planar switch (positive or negative) that involves a pair of distinct vertices $v,w \in V(\Gamma)$.  In each case, there exists a circle $S^1 \subset S^2$ such that $S^1 \cap \Gamma = \{ v,w \}$, and it is a Conway circle for $D$.  Just as in the first case of the forward implication considered above, a flip and the two types of planar switch correspond to three types of Conway mutation with respect to this circle.

Next consider the case of a swap involving a vertex $v$.  In this case, the boundary of a regular neighborhood of the disks $D^2_1, D^2_2$ involved is a Conway circle for $D$.  By an isotopy we may arrange so that the disk $D^2_0$ bounded by this circle is the unit disk and $\{ (t,0,0), -1 \leq t \leq 1\} $ is supported in the black region that contains $v$.  Now a positive swap corresponds to rotating about the $z$-axis, while a negative swap corresponds to rotating about the $x$-axis.  In either case we obtain a mutation.

Lastly, consider the case of a flip involving a single vertex $v$.  In this case, the circle $S^1 \subset S^2$ involved in the flip meets $D$ in a pair of points.  Apply an isotopy of $D$ as in Figure \ref{f: mutation3} to introduce two new intersection points with $S^1$.  A rotation about the $x$-axis followed by an isotopy effects a rotation in the $y$-axis in the original disk, and this corresponds to the flip.  Thus we obtain a mutation in this case as well.

\begin{figure}
\centering
\includegraphics[width=4in]{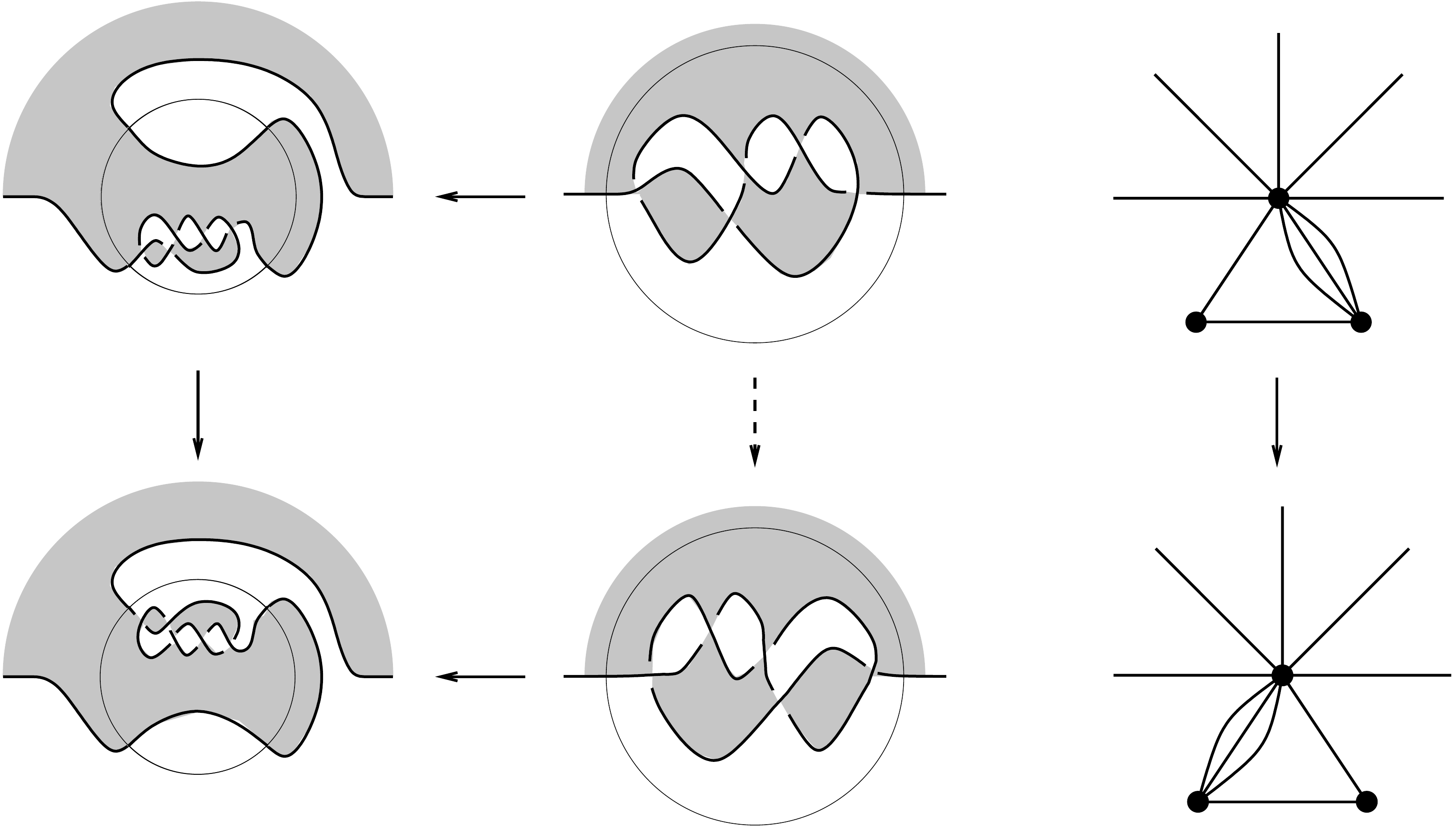}
\put(-197,132){$\sim$}
\put(-263,82){$\rho_x$}
\put(-197,35){$\sim$}
\put(-30,78){\small flip}
\caption{Between mutations and graph operations (3).}  \label{f: mutation3}
\end{figure}
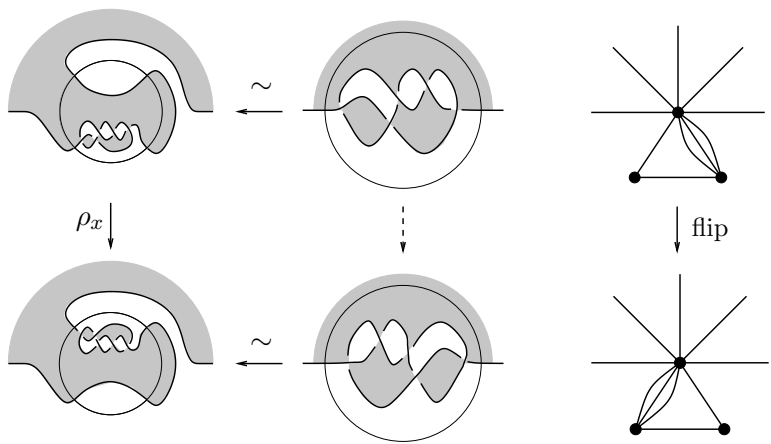

\end{proof}

\begin{cor}[cf. \cite{ck:mutation}, $\S$3, Prop.1]\label{c: mutant graphs}

A pair of connected, reduced alternating diagrams are mutants iff their Tait graphs are 2-isomorphic.

\end{cor}

\begin{proof}

This follows at once from Corollary \ref{c: flip and switch} and Lemma \ref{l: graph mutation}, noting that a diagram is connected and reduced iff its Tait graph is 2-edge-connected.

\end{proof}


\subsection{Two-bridge links.}

We apply Corollary \ref{c: mutant graphs} to derive the result concerning two-bridge links quoted in $\S$\ref{ss: repercussions}.

\begin{prop}\label{p: 2bridge}

A pair of 2-bridge link diagrams in standard position are mutants iff they coincide up to isotopy and reversal \textup{(}i.e. an orientation-reversing homeomorphism of $S^2$\textup{)}.

\end{prop}

\begin{proof}

First observe that the Proposition is obvious in the case that one of the diagrams is the standard diagram for the unknot or two-component unlink.  Excluding these cases, let $D$ denote a 2-bridge link diagram in standard position.  This means that $D$ is connected, reduced, and alternating; its Tait graph $G$ contains a Hamiltonian cycle $H$; every edge in $E(G) - E(H)$ is incident with a fixed vertex $v_0 \in V(G)$; and in the plane drawing $\Gamma$ of $G$, the interiors of all these edges lie in one fixed region of $S^2 - H$.  Note that these conditions specify the image of $\Gamma$ up to isotopy and reversal.

Let $e,f \in E(H)$ denote the edges incident with $v_0$, and let $v_1,\dots,v_k$ denote the neighbors of $v_0$ in $G - \{e,f\}$, chosen with respect to some cyclic order on $H$ (thus, $k=0$ iff $G = H$).  Let $E_i$ denote the set of edges between $v_0$ and $v_i$ in $G - \{e,f\}$, for $i=0,\dots,k+1$, and let $F_i$ denote the edge set of the path directed from $v_i$ to $v_{i+1}$ along $H$, for $i=0,\dots,k$.  Here we take $v_{k+1} = v_0$, so $E_0 = E_{k+1} = \varnothing$, and $F_0 = E(H)$ iff $k=0$.

Now suppose that $D'$ is another 2-bridge link diagram in standard position, and it is a mutant of $D$.  By Corollary \ref{c: mutant graphs},  there exists a 2-isomorphism $\varphi: E(G) \to E(G')$, where $G'$ denotes the Tait graph of $D'$.  We argue that, in fact, $G \cong G'$.

Since $G'$ is 2-edge-connected, it follows that $\varphi(E(H))$ is the edge set of a Hamiltonian cycle $H'$ in $G'$.  Define $v'_0, \dots, v'_{k'+1}$, $E'_0,\dots, E'_{k'+1}$, $F'_0,\dots,F'_{k'}$ with respect to $G'$ as above.  Observe that the edge sets $\varnothing$ and $\bigcup_{t=i}^j F_t$, $i \leq j$, are the intersections between cycles of $G$ with $E(H)$.  Also, any cycle that meets $H$ in $F_i$ uses an edge from $E_i$ and an edge from $E_{i+1}$.  The same applies to $G'$, mutatis mutandis.  It follows that $\varphi$ carries $F_0,\dots,F_k$ to $F'_0, \dots, F'_{k'}$ in (possibly reverse) order, and $E_0,\dots,E_k$ to $E'_0, \dots, E'_{k'}$ in the same order.   In particular, $k = k'$.  Since the $E$'s are sets of parallel edges and the $F$'s are edge sets of paths, it follows that $G \cong G'$.  (Note, however, that this isomorphism does not necessarily induce $\varphi$.)

Since $G \cong G'$ and the images of $\Gamma, \Gamma'$ are unique up to isotopy and reversal, it follows that $D$ and $D'$ coincide up to isotopy and reversal as well.

\end{proof}


\subsection{Heegaard Floer homology.}\label{ss: HF}

We recall here the necessary input from Heegaard Floer homology.  We work with the simplest version of this theory, namely the invariant $\widehat{HF}(Y)$, defined over the field $\bF_2$, for a rational homology sphere $Y$.  The invariant takes the form of a finite-dimensional vector space over $\bF_2$, graded by spin$^c$ structures on $Y$ and rational numbers:
\[ \widehat{HF}(Y) = \bigoplus_{\s \in \Spc(Y)} \widehat{HF}(Y,\s), \quad \widehat{HF}(Y,\s) = \bigoplus_{d \in \Q} \widehat{HF}_d(Y,\s). \]
$\Spc$ structures on $Y$ form a torsor over the group $H^2(Y;\Z)$.  We may therefore regard the invariant as a pair $(\Spc(Y),\widehat{HF}(Y,\cdot))$, where $\widehat{HF}(Y,\cdot)$ takes values in the set of finitely generated, rationally graded vector spaces over $\bF_2$.  From this point of view, a pair of rational homology spheres $Y_1,Y_2$ have isomorphic Heegaard Floer homology groups if there exists an isomorphism $(\Spc(Y_1), \widehat{HF}(Y_1,\cdot)) \overset{\sim}{\to} (\Spc(Y_2), \widehat{HF}(Y_2,\cdot))$.

The invariant has a fundamental non-vanishing property: $\widehat{HF}(Y,\s) \ne 0$, $\forall \, \s \in \Spc(Y)$.  Furthermore, there exists a distinguished grading, the {\em correction term} $d(Y,\s) \in \Q$, with the property that $\widehat{HF}_{d(Y,\s)} (Y,\s) \ne 0$.  By definition, the pair $(\Spc(Y),d(Y,\cdot))$ is the {\em $d$-invariant} of $Y$.  The space $Y$ an {\em L-space} if $\rk \, \widehat{HF}(Y,\s) = 1$, $\forall \, \s \in \Spc(Y)$.  In this case, $\widehat{HF}(Y,\s)$ is supported in the single grading $d(Y,\s)$.  Thus, for an $L$-space $Y$, the isomorphism type of its Heegaard Floer homology groups determines and is determined by that of its $d$-invariant.

\begin{thm}[Ozsv\'ath-Szab\'o \cite{os:doublecover}, Thm.3.4]\label{t: osz}

Let $L$ denote a non-split alternating link and $G$ the Tait graph of an alternating diagram of $L$.  The branched double-cover $\Sigma(L)$ is an L-space, and the {\em (}Heegaard Floer{\em )} $d$-invariant of $\Sigma(L)$ is isomorphic to {\em minus} the {\em (}lattice theoretic{\em )} $d$-invariant of $\cF(G)$. \qed

\end{thm}

We arrive at last to the proof of our main topological result.

\begin{proof}[Proof of Theorem \ref{t: main}.]

(1)$\implies$(2) is clear, (2)$\implies$(3) is the observation of Viro, and (3)$\implies$(4) is clear.  It stands to establish (4)$\implies$(1).  By Theorem \ref{t: osz}, the $d$-invariants of $\cF(G)$ and $\cF(G')$ are isomorphic, where $G, G'$ denote the Tait graphs of $D, D'$.  By Theorem \ref{t: main graph}, it follows that $G$ and $G'$ are 2-isomorphic, so by Corollary \ref{c: mutant graphs}, it follows that $D$ and $D'$ are mutants.  This completes the proof of the Theorem.
\end{proof}

\bibliographystyle{myalpha}
\bibliography{/Users/Josh/Desktop/Papers/References}

\providecommand{\bysame}{\leavevmode\hbox to3em{\hrulefill}\thinspace}
\providecommand{\MR}{\relax\ifhmode\unskip\space\fi MR }
\providecommand{\MRhref}[2]{%
  \href{http://www.ams.org/mathscinet-getitem?mr=#1}{#2}
}
\providecommand{\href}[2]{#2}
\begin{thebibliography}{BdlHN97}

\bibitem[BdlHN97]{bacheretal:lattices}
R.~Bacher, P.~de~la Harpe, and T.~Nagnibeda, \emph{The lattice of integral
  flows and the lattice of integral cuts on a finite graph}, Bull. Soc. Math.
  France \textbf{125} (1997), no.~2, 167--198.

\bibitem[Bon83]{bonahon:lens}
F.~Bonahon, \emph{Diff\'eotopies des espaces lenticulaires}, Topology
  \textbf{22} (1983), no.~3, 305--314.

\bibitem[Bro60]{brody:torsion}
E.~J. Brody, \emph{The topological classification of the lens spaces}, Ann. of
  Math. (2) \textbf{71} (1960), 163--184.

\bibitem[CK08]{ck:mutation}
Abhijit Champanerkar and Ilya Kofman, \emph{On mutation and {K}hovanov
  homology}, Commun. Contemp. Math. \textbf{10} (2008), no.~suppl. 1, 973--992.

\bibitem[CS99]{cs:lattices}
J.~H. Conway and N.~J.~A. Sloane, \emph{Sphere packings, lattices and groups},
  third ed., vol. 290, Springer-Verlag, New York, 1999.

\bibitem[Dun10]{dunfield:comm}
N.~Dunfield, \emph{Private communication}.

\bibitem[Elk95]{elkies}
N.~D. Elkies, \emph{A characterization of the {$\mathbb{Z}^n$} lattice}, Math.
  Res. Lett. \textbf{2} (1995), no.~3, 321--326.

\bibitem[GR01]{godsilroyle:book}
C.~Godsil and G.~Royle, \emph{Algebraic graph theory}, Graduate Texts in
  Mathematics, vol. 207, Springer-Verlag, New York, 2001.

\bibitem[Hak65]{hakimi}
S.~L. Hakimi, \emph{On the degrees of the vertices of a directed graph}, J.
  Franklin Inst. \textbf{279} (1965), 290--308.

\bibitem[HR85]{hr:2bridge}
C.~Hodgson and J.~H. Rubinstein, \emph{Involutions and isotopies of lens
  spaces}, Knot theory and manifolds ({V}ancouver, {B}.{C}., 1983), Lecture
  Notes in Math., vol. 1144, Springer, Berlin, 1985, pp.~60--96.

\bibitem[Kau87]{kauff:altknots}
L.~H. Kauffman, \emph{State models and the {J}ones polynomial}, Topology
  \textbf{26} (1987), no.~3, 395--407.

\bibitem[Kaw96]{kawauchi:book}
A.~Kawauchi, \emph{A survey of knot theory}, Birkh\"auser Verlag, Basel, 1996.

\bibitem[Kir10]{kirby:problems}
R.~Kirby, \emph{Problems in low-dimensional topology}, {\tt
  math.berkeley.edu/\~{}kirby/problems.ps.gz} (2010).

\bibitem[Men84]{menasco:alternating}
W.~Menasco, \emph{Closed incompressible surfaces in alternating knot and link
  complements}, Topology \textbf{23} (1984), no.~1, 37--44.

\bibitem[MT91]{mt:tait}
W.~Menasco and M.~B. Thistlethwaite, \emph{The {T}ait flyping conjecture},
  Bull. Amer. Math. Soc. (N.S.) \textbf{25} (1991), no.~2, 403--412.

\bibitem[MT01]{mt:book}
B.~Mohar and C.~Thomassen, \emph{Graphs on surfaces}, Johns Hopkins Studies in
  the Mathematical Sciences, Johns Hopkins University Press, Baltimore, MD,
  2001.

\bibitem[Mur87]{murasugi:altknots}
K.~Murasugi, \emph{Jones polynomials and classical conjectures in knot theory},
  Topology \textbf{26} (1987), no.~2, 187--194.

\bibitem[OS05]{os:definite}
B.~Owens and S.~Strle, \emph{Definite manifolds bounded by rational homology
  three spheres}, Geometry and topology of manifolds, Fields Inst. Commun.,
  vol.~47, Amer. Math. Soc., Providence, RI, 2005, pp.~243--252.

\bibitem[OSz03]{os:absgr}
P.~Ozsv{\'a}th and Z.~Szab{\'o}, \emph{Absolutely graded {F}loer homologies and
  intersection forms for four-manifolds with boundary}, Adv. Math. \textbf{173}
  (2003), no.~2, 179--261.

\bibitem[OSz05]{os:doublecover}
\bysame, \emph{On the {H}eegaard {F}loer homology of branched double-covers},
  Adv. Math. \textbf{194} (2005), no.~1, 1--33.

\bibitem[Rei35]{reidemeister:torsion}
K.~Reidemeister, \emph{Homotopieringe und linsenr\"aume}, Abh. Math. Sem. Univ.
  Hamburg \textbf{11} (1935), 102--109.

\bibitem[Rus05]{rustamov:turaev}
R.~Rustamov, \emph{Surgery formula for the renormalized {E}uler characteristic
  of {H}eegaard {F}loer homology}, {\tt math.GT:0409294} (2005).

\bibitem[Sch56]{schubert:2bridge}
H.~Schubert, \emph{Knoten mit zwei {B}r\"ucken}, Math. Z. \textbf{65} (1956),
  133--170.

\bibitem[Sch93]{schrijver:tait}
A.~Schrijver, \emph{Tait's flyping conjecture for well-connected links}, J.
  Combin. Theory Ser. B \textbf{58} (1993), no.~1, 65--146.

\bibitem[Sch03]{schrijver:book}
\bysame, \emph{Combinatorial optimization. {P}olyhedra and efficiency},
  Algorithms and Combinatorics, vol.~24, Springer-Verlag, Berlin, 2003.

\bibitem[Thi87]{thistle:spanning}
M.~B. Thistlethwaite, \emph{A spanning tree expansion of the {J}ones
  polynomial}, Topology \textbf{26} (1987), no.~3, 297--309.

\bibitem[Tru80]{truemper:switch}
K.~Truemper, \emph{On {W}hitney's {$2$}-isomorphism theorem for graphs}, J.
  Graph Theory \textbf{4} (1980), no.~1, 43--49.

\bibitem[vdB59]{vanderblij:characteristic}
F.~van~der Blij, \emph{An invariant of quadratic forms mod {$8$}}, Nederl.
  Akad. Wetensch. Proc. Ser. A 62 = Indag. Math. \textbf{21} (1959), 291--293.

\bibitem[Vir76]{viro:mutation}
O.~Ja. Viro, \emph{Nonprojecting isotopies and knots with homeomorphic
  coverings}, Zap. Nau\v cn. Sem. Leningrad. Otdel. Mat. Inst. Steklov. (LOMI)
  \textbf{66} (1976), 133--147, 207--208, Studies in topology, II.

\bibitem[Whi33]{whitney:switch}
H.~Whitney, \emph{2-{I}somorphic graphs}, Amer. J. Math. \textbf{55} (1933),
  no.~1-4, 245--254.

\end{thebibliography}

\end{document}